\documentclass[a4paper,11pt]{article}

\usepackage[left=3.1cm, right=3.1cm, top=3cm, bottom=2.75cm]{geometry}

\usepackage[english]{babel}
\usepackage[utf8]{inputenc}
\usepackage[T1]{fontenc}

\usepackage{amsfonts}
\usepackage{mathrsfs}
\usepackage{bbm}
\usepackage{xcolor}
\usepackage{latexsym}
\usepackage{mathdots}
\usepackage{amssymb}
\usepackage{mathtools}
\usepackage{wasysym}

\usepackage{enumitem}
\setlist{nolistsep}
\usepackage{amsthm}
\usepackage[backend=biber,
    bibencoding=utf8,
    style=numeric, 
    giveninits=true,
    isbn=false,
    doi=true,
    url=false,
    clearlang=true,
    defernumbers=true,
]{biblatex}
\usepackage{csquotes}
\usepackage{xcolor}
\usepackage[linktocpage=true]{hyperref}
\usepackage[capitalize]{cleveref}
\hypersetup{citecolor = black,colorlinks,
    		linkcolor = black,
			urlcolor = blue}

\newtheoremstyle{plain}{3mm}{3mm}{\slshape}{}{\bfseries}{.}{.5em}{}
\newtheoremstyle{claim}{3mm}{3mm}{}{}{\itshape}{.}{.5em}{}
\newtheoremstyle{definition}{2mm}{2mm}{}{}{\bfseries}{.}{.5em}{}
\theoremstyle{plain}
	
\newtheorem{theorem}{Theorem}
\numberwithin{theorem}{section}
\newtheorem{lemma}[theorem]{Lemma}
\newtheorem{proposition}[theorem]{Proposition}

\theoremstyle{claim}

\theoremstyle{definition}

\newtheorem{remark}[theorem]{Remark}

\newcommand{\N}{\mathbb{N}}
\newcommand{\Z}{\mathbb{Z}}
\newcommand{\R}{\mathbb{R}}
\newcommand{\C}{\mathbb{C}}

\newcommand{\T}{\mathbb{T}}
\renewcommand{\P}{\mathbb{P}}  
\newcommand{\one}{\boldsymbol{1}}  

\title{A Joint Local Limit Theorem for Prime Divisor Counting Functions}
\author{Jovan Andreevski}
\date{\today}
\begin{document}
\maketitle

\begin{abstract}
\noindent
We establish a uniform two-dimensional local limit theorem for the joint distribution of $\Omega(n)$ and $\Omega(n+1)$. More precisely, along a set of scales of logarithmic density one, the normalized joint counts are uniformly approximated by the density of two independent Gaussian random variables, each with mean and variance $\log\log x$, up to a power-saving error. The proof builds upon a recent circle-method framework of Tao and Teräväinen. Our results equally hold for the additive function $\omega$. Additionally, we explore applications of our local limit theorem in the setting of uniquely ergodic systems.
\end{abstract}
\tableofcontents

\section{Introduction}\label{sec:introduction}

A central theme in probabilistic number theory is the study of the statistical behaviour of additive functions. Among the most extensively studied additive functions are the function counting the number of distinct prime divisors of a natural number $n$ 
\begin{align}
    \omega(n) := \sum_p \one_{p|n},
\end{align}
alongside the function counting the number of prime divisors of $n$ with multiplicity
\begin{align}
    \Omega(n) := \sum_{p} \sum_{j \geq 1} \one_{p^j|n} = \sum_{p} v_p(n),
\end{align}
where $v_p(n) := \sum_{j \geq 1} \one_{p^j|n}$ is the \(p\)-adic valuation of $n$.

Although these functions fluctuate considerably, their values obey remarkably regular statistical laws. Hardy and Ramanujan~\cite{HardyRamanujan1917} proved that both have normal order $\log_2n := \log\log n$, meaning that for every $\varepsilon >0$ and $\nu \in \{\omega, \Omega\}$,
\begin{align}
\frac{1}{x} \cdot
\#\left\{
n\leq x:
\left|\nu(n)-\log_2n\right|
>
\varepsilon\log_2n
\right\}
\longrightarrow 0,
\qquad
\text{as} \,\, x \to \infty.
\end{align}
The Erd\H{o}s--Kac theorem~\cite{ErdosKac1940} identifies the fluctuations around
this normal order. The theorem shows that, for every
$t\in\R$ and $\nu\in\{\omega,\Omega\}$,
\begin{align}
\label{eq:erdos-kac}
 \frac1x \cdot
 \#\left\{n\leq x:\frac{\nu(n)-\log_2x}{\sqrt{\log_2x}}\leq t\right\}\
  \to \frac1{\sqrt{2\pi}}\int_{-\infty}^{t}e^{-u^2/2}\,\mathrm du,
\qquad
\text{as} \,\, x \to \infty.
\end{align}
Therefore, the Erd\H{o}s--Kac theorem shows that, after centering and normalization, $\omega(n)$ and $\Omega(n)$ have the same Gaussian limiting behaviour as a sum of genuinely independent random variables.


The Erd\H{o}s–Kac theorem was subsequently refined and extended in several directions. In their original paper, Erd\H{o}s and Kac already established \cref{eq:erdos-kac} for a wide range of additive functions, and subsequently, Kubilius \cite{Kubilius1955} and Shapiro \cite{Shapiro1956} extended it to an even broader class of strongly additive functions. Rényi and Turán \cite{RenyiTuran1958} obtained the optimal rate of convergence $O((\log_2 x)^{-1/2})$ in the classical Erd\H{o}s--Kac theorem. In a different direction, Erd\H{o}s \cite{Erdos1948} studied the number of integers having exactly $m$ distinct prime factors, obtaining asymptotic formulas valid uniformly for $m$ in the central range. Subsequently, Sathe and Selberg \cite{Sathe1953I,Sathe1953II,Sathe1954III,Sathe1954IV,Selberg1954} established more general uniform asymptotic formulas, valid over wider ranges of $m$. For a systematic account of these and related developments concerning the distribution of additive functions, we refer the reader to \cite{GranvilleSoundararajan2007, Kubilius1964, Tenenbaum2015}. 

A natural follow-up question concerns the joint statistical behaviour of $\nu(n)$ and $\nu(n+1)$. Due to the distinct multiplicative structures of consecutive integers, one expects the normalized quantities $\widetilde{\nu}_x(n) = (\nu(n) - \log_2x)/ \sqrt{\log_2x}$ and $\widetilde{\nu}_x(n+1) = (\nu(n+1) - \log_2x)/ \sqrt{\log_2x}$ to become asymptotically independent as $x\to\infty$, suggesting a bivariate analogue of the Erd\H{o}s--Kac theorem. Such a result was first established by LeVeque \cite{LeVeque1949} for $\omega$, who proved convergence to a pair of independent standard Gaussian random variables, albeit without an explicit quantitative rate of convergence. Subsequently, Tanaka generalised LeVeque's result by establishing a multivariate Erd\H{o}s–Kac theorem for the joint distribution of $(\omega(\left|P_1(n)\right|),\ldots,\omega(\left|P_r(n)\right|))$, where $P_1(t),\ldots,P_r(t) \in \Z[t]$ are pairwise coprime irreducible polynomials \cite{Tanaka1955}. Quantitative estimates for such multivariate distributions were later obtained by Eda and Yamano \cite{EdaYamano1969}. More recently, Mangerel \cite{Mangerel2020} established a quantitative bivariate Erd\H{o}s--Kac theorem for $(\omega(n),\omega(n+1))$. See \cite{Kubilius1964} for an account of more general results concerning the joint distribution of shifted additive functions.  

The Erd\H{o}s--Kac theorem can be viewed as a central limit theorem for $\omega(n)$ and the results of LeVeque and Mangerel as bivariate central limit theorems for the pair $(\omega(n), \omega(n+1))$. Similarly, we can interpret the results of Erd\H{o}s, Sathe and Selberg as local limit theorems for $\omega(n)$. Naturally, this leads to the question whether we can establish a bivariate local limit theorem for $(\omega(n), \omega(n+1))$. More precisely, we seek a uniform asymptotic formula for the normalized joint counts
\begin{align}
    \label{eq:joint-local-frequency}
P_x^{\nu_0,\nu_1}(\mathbf m):= \frac1x \cdot 
\#\{n\leq x:
\nu_0(n)=m_0,\ \nu_1(n+1)=m_1\},
\end{align}
where $\mathbf{m} = (m_0, m_1) \in \Z^2$, and $\nu_0,\nu_1\in\{\omega,\Omega\}$.


Our main theorem answers this question affirmatively with a power-saving error term. For $t>0$ and $\mathbf m=(m_0,m_1)\in\Z^2$, set
\begin{equation}\label{eq:intro-Gaussian}
 G_t(\boldsymbol m)
 :=\frac1{2\pi t}
 \exp\!\left(-\frac{(m_0-t)^2+(m_1-t)^2}{2t}\right).
\end{equation}

\begin{theorem}[Joint Local Limit Theorem]\label{thm:intro-main}
There exist an absolute constant $c>0$ and a set $\mathcal X\subset\N$ such that
\begin{equation}\label{eq:intro-exceptional}
 \frac1{\log X}\sum_{\substack{n\leq X\\n\notin\mathcal X}}\frac1n
 \ll (\log\log X)^{-c}
\end{equation}
for all sufficiently large $X$, and
\begin{equation}\label{eq:intro-main}
 P_x^{\nu_0,\nu_1}(\mathbf m)
 =G_{\log_2x}(\boldsymbol{m})+O((\log_2x)^{-1-c})
\end{equation}
uniformly for all $x\in\mathcal X$, all $\boldsymbol{m} = (m_0,m_1)\in\Z^2$, and all $\nu_0,\nu_1\in\{\omega,\Omega\}$.
\end{theorem}
We note that, in the central range, the Gaussian main term is of order $(\log_2 x)^{-1}$, so the error term in Theorem \ref{thm:intro-main} represents a power saving relative to the main term. Our proof allows us to take $c = 1/16$.

The proof of our result builds upon a recent advance due to Tao and Ter\"av\"ainen~\cite{TaoTeravainen2025} on conjectures of Erd{\H{o}}s, Pomerance and S{\'a}rk{\"o}zy \cite{ErdosSarkozyPomerance1985} and of Hildebrand
\cite{Hildebrand1987}, concerning the proportion of consecutive integers having the same number of prime factors (with or without multiplicity) or the same number of divisors. In \cite[Theorem 1.7]{TaoTeravainen2025}, they show that there exists a set $\mathcal X\subset\mathbb N$ of logarithmic density one such that, for all $x\in\mathcal X$, one has
\begin{equation}\label{eq:intro-TT}
 \frac1x \cdot \#\{n\leq x:\nu(n)=\nu(n+1)\}
 =\frac{1+O((\log_2x)^{-c})}{2\sqrt{\pi\log_2x}},
 \qquad \nu\in\{\omega,\Omega\},
\end{equation}
for some absolute constant \(c>0\). Their theorem also establishes an analogous result for the divisor function $\tau$. We remark that in \cite[Remark 4.2]{TaoTeravainen2025}, Tao and Ter\"av\"ainen suggest the possibility that their approach could be extended to obtain a two-dimensional local limit theorem as the one we establish.

The approach in \cite{TaoTeravainen2025} uses the circle method to separate probabilities such as $\mathbf{P}(\Omega(\mathbf{n}) = \Omega(\mathbf{n+1}))$ into a "low-frequency" and a "high-frequency" term. The low-frequency term is handled using sieve theory and a probabilistic model for small prime factors, whereas the high-frequency term is controlled by a new quantitative correlation estimate for multiplicative functions \cite[Theorem 3.1]{TaoTeravainen2025}. 

Our proof broadly follows the same approach. In our case, handling the high-frequency term is a rather straightforward technical adaptation of \cite{TaoTeravainen2025}. The main contribution occurs in the handling of the low-frequency term, where we establish a bivariate local limit theorem that handles all four possible pairs $(\nu_0, \nu_1) \in \{\omega, \Omega\}^2$. Moreover, we require several quantitative refinements in order to ensure precision in the error term of the required $(\log_2x)^{-1}$ scale. 

Theorem~\ref{thm:intro-main} relates to and extends several previous results. Recall that we write $P_x^{\nu_0,\nu_1}(\mathbf m)$ for the normalized joint counts in \cref{eq:joint-local-frequency}. 

First, we notice that our result recovers the result of Tao and Ter\"av\"ainen. Taking $m_1 = m_0 = m$ in our theorem and summing, we obtain for $\nu \in \{\omega, \Omega\}$
\begin{align*}
    \frac1x \cdot \#\{n\leq x:\nu(n)=\nu(n+1)\} = \sum_{m \geq 0} P_x^{\nu,\nu}(m, m) 
\end{align*}
Splitting the sum in the central range $\left| m - \log_2x \right| < (\log_2x)^{\frac{c+1}{2}}$ and its complement, where $0 < c < 1/2$ is the constant supplied by Theorem \ref{thm:intro-main}, yields
\begin{align*}
    \sum_{m \geq 0} P_x^{\nu,\nu}(m, m) = \frac{1}{2\pi \log_2x}
\sum_{m\in\mathbb{Z}} e^{-(m-\log_2x)^2/\log_2x}
+ O\left((\log_2x)^{-1/2-c/2}\right),
\end{align*}
where the contribution from the complimentary range is shown to be negligible by higher-moment estimates of $\omega$ and $\Omega$ \cite[Theorem 1]{GranvilleSoundararajan2007}, and the use of Markov's inequality. Finally, evaluating the sum using Poisson summation (see Section \ref{subsec:prelim-Fourier}) gives the asymptotic formula in \cref{eq:intro-TT}. 

Furthermore, the argument for consecutive integers can be adapted without significant modifications to obtain a local limit theorem for $n$ and $n+a$, with $a$ being a fixed shift. Theorem~\ref{thm:intro-main} also provides a local counterpart to the results of LeVeque and Mangerel. Related joint upper bounds of the correct order of magnitude were obtained by Goudout \cite[Th\'eor\`eme~3]{Goudout2017}. In the central range, our theorem sharpens these bounds to a uniform asymptotic formula with an explicit Gaussian main term and a power-saving relative error.


Our result has consequences for asymptotic independence as well.
Charamaras and Richter
\cite[Theorem~A and Appendix~A]{CharamarasRichter2024} proved that, for any bounded functions $a,b:\mathbb N\to\mathbb C$,
\begin{align*}
 \frac1{\log N}\sum_{n\le N}
 \frac{a(\Omega(n))b(\Omega(n+1))}{n}
 =\left(\frac1N\sum_{n\le N}a(\Omega(n))\right)
  \left(\frac1N\sum_{n\le N}b(\Omega(n))\right)+o(1).
\end{align*}
In fact, their work provides quantitative versions of this assertion and yields the optimal error term $O ((\log_2N)^{-1/2})$. A subsequent result of Wang \cite[Theorem~1.9]{Wang2026Averaged} generalises this result to bounded functions $a :\mathbb N^2\to\mathbb C$, with a weaker error term. Our theorem recovers both of these results. In fact, we are able to obtain an analogous asymptotic independence result along Cesàro averages, restricted to a set of scales $\mathcal{X}$ of logarithmic density one. More precisely, for every bounded function $a:\mathbb N^2\to\mathbb C$, we have
\begin{align*}
    \frac{1}{N}\sum_{n\le N} a(\nu_0(n),\nu_1(n+1))
=
\frac{1}{N^2}\sum_{n_1, n_2\le N} a(\nu_0(n_1),\nu_1(n_2))
+o(1),
\end{align*}
as $N \to \infty$ along $\mathcal{X}$. In fact, Theorem \ref{thm:intro-main} implies the asymptotic independence of $\nu_0(n)$ and $\nu_1(n+1)$ in total variation along $\mathcal X$.

Finally, we discuss the dynamical consequences of our result. In \cite[Theorem~A]{BergelsonRichter2022}, Bergelson and Richter proved that, for a uniquely ergodic additive topological dynamical system $(Y,\mu,T)$,
\begin{align*}
 \lim_{N\to\infty}\frac1N\sum_{n\le N}F(T^{\Omega(n)}y)
 =\int_Y F\,d\mu
 \qquad (F\in C(Y),\ y\in Y).
\end{align*}
More recently, Wang \cite[Theorem~1.10]{Wang2026Averaged} established a two-point analogue along logarithmic averages, replacing the iterates $T^{\Omega(n)}$ by $T^{\Omega(n)+\Omega(n+1)}$. 
As a further consequence of Theorem \ref{thm:intro-main}, we are able to establish the following dynamical result, which can be seen as a natural continuation of the theorem of Bergelson and Richter, with ordinary convergence restricted to $\mathcal X$.

\begin{theorem}
\label{thm:joint-dynamical-limit}
Let $\mathcal X$ be the set obtained in 
Theorem~\ref{thm:intro-main}. Let
$(Y,\mu,T_0,T_1)$ be a uniquely ergodic $\mathbb N^2$-dynamical system.
Then, for every $F\in C(Y)$,
\begin{align}
\label{eq:intro-joint-dynamics}
\max_{\nu_0,\nu_1\in\{\omega,\Omega\}}
\sup_{y\in Y}
\left|
\frac1N\sum_{n\le N}
F\bigl(T_0^{\nu_0(n)}T_1^{\nu_1(n+1)}y\bigr)
-\int_Y F\,d\mu
\right|\longrightarrow 0
\end{align}
as $N\to\infty$ with $N\in\mathcal X$.
Moreover, without restricting the scales,
\begin{align}
\label{eq:intro-joint-log-dynamics}
\max_{\nu_0,\nu_1\in\{\omega,\Omega\}}
\sup_{y\in Y}
\left|
\frac1{\log N}\sum_{n\le N}
\frac{F\bigl(T_0^{\nu_0(n)}T_1^{\nu_1(n+1)}y\bigr)}{n}
-\int_Y F\,d\mu
\right|\longrightarrow 0
\end{align}
as $N\to\infty$.
\end{theorem} The background on uniquely ergodic systems and the proof of Theorem~\ref{thm:joint-dynamical-limit} are given in
Section~\ref{sec:dynamical}.


We note that removing the restriction $N\in\mathcal X$ from \eqref{eq:intro-joint-dynamics} would have a substantial arithmetic
consequence. Indeed, let $Y=\mathbb Z/2\mathbb Z$ be equipped with the uniform probability measure, take both $T_0$ and $T_1$ to
be the rotation $y\mapsto y+1$, and set $F(y)=(-1)^y$. Setting $y=0$ and $\nu_0=\nu_1=\Omega$ in \eqref{eq:intro-joint-dynamics}  would entail a proof of the two-point ordinary Chowla conjecture. On the other hand, notice that the logarithmic conclusion \eqref{eq:intro-joint-log-dynamics}, with the same choice of system, recovers the two-point logarithmically averaged Chowla theorem, first established by Tao \cite{Tao2016LogChowla}.

The rest of the paper is organized as follows. In section \ref{sec:preliminaries} we present the necessary preliminaries for the main proof. In section \ref{sec:circle-method}, we reduce Theorem \ref{thm:intro-main} to an averaged dyadic statement (Theorem \ref{thm:intro-dyadic}), introduce the circle-method decomposition into low and high frequencies, and establish the required high-frequency estimate. In section \ref{sec:low-frequencies} we prove the complementary low-frequency estimate. Finally, in section \ref{sec:dynamical} we give the proof of our main dynamical result, Theorem \ref{thm:joint-dynamical-limit}.\\

\noindent \textbf{Acknowledgments:} The author thanks Florian Richter for comments on an early draft of the article and for his guidance during the preparation of this article.

\noindent \textbf{Use of technology:} No AI tools were used to produce any of the results or the ideas in this article. The only instance of AI use was to create a referee-style report of the article before submission. For this we used GPT-6 Extra High, and only minor corrections were found. The author assumes responsibility for all content.

\section{Preliminaries}\label{sec:preliminaries}

In this section, we collect the notation and needed preliminaries from analytic number theory, Fourier analysis and \cite{TaoTeravainen2025} which are used in the proof of Theorem \ref{thm:intro-dyadic}. 

\subsection{Notation and Conventions}\label{subsec:prelim-notation}

We write $\N=\{1,2,\ldots\}$, $\N_0=\N\cup\{0\}$ and $\P$ for the primes. We write $\one_A$ and $\#A$ for the indicator and cardinality of a set $A$, respectively.  We
make use of the iterated logarithm notation
\[
 \log_1 x=\log x,\quad \log_{j+1}x=\log(\log_jx),
\]
where we take the argument of these iterated logarithms large enough so that the relevant expressions exist and are positive. 

Let $\T=\R/\Z$, identified with $[-1/2,1/2)$ and equipped with the normalised Haar measure. We write $e(t)=e^{2\pi it}$ and $\|t\|_{\T}=\min_{k\in\Z}|t-k|$. We use boldface notation for vectors. All vector norms are assumed to be Euclidean. For a scalar-, vector- or matrix-valued function $f$, we write $\|f\|_\infty$ to denote its sup norm, using the operator norm in the matrix-valued case. For a sufficiently differentiable function $f:\mathbb R\to \C$, we write \(f^{(j)}\) for its \(j\)-th derivative. For functions of several variables, we use \(\partial_{x_i}\), \(\nabla\), and \(D^2\) for partial derivatives, the gradient, and the Hessian, respectively.

Probability, expectation and variance are denoted by $\mathbf P$, $\mathbf E$ and $\mathbf{Var}$, respectively. For a random vector $\boldsymbol{S}$, we write $\boldsymbol{\operatorname{Cov}}(\boldsymbol{S})$ for its covariance matrix and $\Phi_{\boldsymbol{S}}(\boldsymbol\alpha)=\mathbf{E} e(\boldsymbol\alpha\cdot \boldsymbol{S})$ for its characteristic function.

In terms of asymptotic notation, the relations $X=O(Y)$ and $X\ll Y$ mean $|X|\leq CY$ for some constant $C$, $X=o(Y)$ means $|X|\leq c(x)Y$ for some $c(x) \rightarrow 0$ as $x \to \infty$, $X\asymp Y$ means $X\ll Y\ll X$. Note that writing subscripts in asymptotic notation indicates permitted dependencies. 

We use $\varphi$, $\mu$ and $\tau$ for Euler's totient function, the M\"obius function and the divisor function, respectively.  We write $d \mid n$ to denote that $d$ divides $n$. Additionally, we write $\operatorname{rad}(n)=\prod_{p\mid n}p$, and $p^a \parallel n$ when the $p$-adic valuation of $n$ equals $a$. 

For $z\geq2$, let $\N_{\leq z}$ and $\N_{>z}$ be the multiplicative
semigroups generated respectively by the primes $p\leq z$ and $p>z$, including
the empty product.  Every $n$ factors uniquely as
\begin{align*}
 n=n_{\leq z}n_{>z},\quad \text{where} \, \, \,
 n_{\leq z}=\prod_{p\leq z}p^{v_p(n)},
\end{align*}
and, for $\nu\in\{\omega,\Omega\}$, we write
\begin{align*}
    \nu_{\leq z}(n):=\nu(n_{\leq z}),\quad
    \nu_{>z}:=\nu-\nu_{\leq z}.
\end{align*}
Let $(a, b)$ denote the greatest common divisor of $a$ and $b$. An arithmetic function $f$ is multiplicative if $f(1)=1$ and $f(ab)=f(a)f(b)$ whenever $(a, b)=1$, and is $1$-bounded if $|f|\leq1$.
For $1$-bounded multiplicative functions $f,g$, and a scale $Y \geq 1$, define the pretentious distance
and the associated minimisation quantity by
\begin{align*}
 D(f,g;Y)^2:=\sum_{p\leq Y}\frac{1-\Re(f(p)\overline{g(p)})}{p},\qquad
 M(f;Y,Q):=\inf_{\substack{|t|\leq Y,\ q\leq Q\\ \chi\ ({\rm mod}\ q)}}
 D(f,\chi(\cdot)(\cdot)^{it};Y)^2.
\end{align*}
The infimum is taken  over all Dirichlet characters $\chi$ of modulus $ q \leq Q$.

\subsection{Preliminaries from Analytic Number Theory}\label{subsec:prelim-sieve}

We repeatedly use the following estimates due to Mertens.
\begin{theorem}\label{thm:Mertens}
    Let $z \geq 3$.
    \vspace{2mm}
    \begin{enumerate}[label=\textup{(\roman*)}]
        \item It holds that
        \begin{equation}
            \sum_{p\leq z}\frac1p=\log_2z+M+O((\log z)^{-1})
        \end{equation} for some constant $M.$ \\
        \item It holds that
        \begin{equation}
            \sum_{p\leq z}\frac{\log p}{p}=\log z+O(1).
        \end{equation}
    \end{enumerate}
\end{theorem}
Another standard result that we make use of is Rankin's trick which essentially stems from the following elementary inequality
\begin{align}
    \one_{\{u\geq y\}}
    \leq \left(\frac{u}{y}\right)^\delta
    \qquad \text{for } u\geq 0,\; y,\delta>0.
\end{align}
In particular, for non-negative arithmetic functions $f$ and for nonnegative random variables $U$, this yields the following two very useful inequalities
\begin{equation}\label{eq:prelim-rankin}
    \sum_{n\geq y} f(n)
    \leq y^{-\delta}\sum_n f(n)n^\delta,
    \qquad
    \mathbf{P}(U\geq y)
    \leq y^{-\delta}\mathbf{E}U^\delta.
\end{equation}
An additional useful estimate, for $z \geq e^2$ and $\delta = 1/\log z$, is
\begin{align}
    \prod_{p\leq z}\frac{1-p^{-1}}{1-p^{\delta-1}}\ll 1,
    \label{euler-product}
\end{align}
which is a straightforward consequence of the second estimate in Mertens' theorem and the elementary estimate $p^\delta-1\ll \delta\log p$, valid for \(p\leq z\).

We also use the fundamental lemma of sieve theory \cite[Theorem~6.12]{FriedlanderIwaniec2010} (see also \cite[Theorem~I.4.4]{Tenenbaum2015}).
The formulation below is deliberately stated such that the sieve-dimension condition is required only up to a fixed multiplicative constant $K$ and with a weaker exponential error, as these are sufficient for our purposes. To that end, let $\mathcal A$ be a finite set of integers and $\mathcal P \subset \P$ a set of primes, and set
\begin{align*}
    P(z):=\prod_{\substack{p\leq z\\ p \in \mathcal{P}}}p,
    \quad
    \mathcal A_d:=\#\{a\in\mathcal A:d\mid a\},
    \quad
    S(\mathcal A, \mathcal P, z) = \#\{a \in \mathcal A:(a, P(z)) = 1\}.
\end{align*}
\begin{lemma}[Fundamental Lemma of Sieve Theory]
    Suppose that there exist $X>0$ and a multiplicative function
$g:\mathbb N\to[0,1]$ such that
\begin{align}
    \mathcal{A}_d = Xg(d)+r_d
    \label{eq:sieve-density}
\end{align}
for every $d\mid P(z)$. Suppose moreover that, for some fixed
$\kappa>0$ and $K\geq1$,
\begin{align}
    \prod_{\substack{w\leq p<v\\ p\in\mathcal P}}
    (1-g(p))^{-1}
    \leq
    K\left(\frac{\log v}{\log w}\right)^\kappa
    \label{eq:sieve-dimension}
\end{align}
for all $2\leq w<v\leq z$. 

Let $D=z^s$, where $s$ is sufficiently large in terms of $\kappa$ and $K$. Then there exists a constant
$c=c(\kappa,K)>0$ such that
\begin{align}
    S(\mathcal A,\mathcal P,z)
    &=
    X\prod_{\substack{p\leq z\\ p\in\mathcal P}}
    (1-g(p))
    \left(1+O_{\kappa,K}(e^{-cs})\right)
    +O\left(
        \sum_{\substack{d\leq D\\ d\mid P(z)}} |r_d|
    \right).
    \label{eq:fundamental-sieve}
\end{align}
\end{lemma}

\subsection{Preliminaries from Fourier Analysis} \label{subsec:prelim-Fourier}

Let $C_c^\infty(\mathbb R)$ denote the space of infinitely differentiable functions on $\R$ with compact support, and let \(\mathcal S(\mathbb R)\) denote the Schwartz space on $\R$. 
For $f\in C_c^\infty(\mathbb R)$, we denote its Fourier transform by $\widehat f$. We make use of \cite[Proposition 6.2.1]{SteinShakarchi2003}, which states that, for every \(f\in C_c^\infty(\mathbb R)\), \(A\ge0\), and \(j\in\mathbb N_0\),
\begin{align}
\bigl|\widehat f^{(j)}(\xi)\bigr|
\ll_{f,A,j}
(1+|\xi|)^{-A}.
\label{eq:fourier-schwartz-decay}
\end{align}
For \(f\in\mathcal S(\mathbb R)\), we shall use Fourier inversion in the form
\cite[Theorem~6.2.4]{SteinShakarchi2003}
\begin{align}
f(u) = \int_{\mathbb R}\widehat f(\xi)e(u\xi)\,d\xi,
\label{eq:fourier-inversion}
\end{align}
as well as the Poisson summation formula
\cite[Theorem~5.3.1]{SteinShakarchi2003}
\begin{align}
\sum_{n\in\mathbb Z}f(n) =
\sum_{n\in\mathbb Z}\widehat f(n),
\qquad f\in\mathcal S(\mathbb R).
\label{eq:poisson-summation}
\end{align}
Note that, for $\mathbf{S}$ being a $\Z^2$-valued random vector, for every \(\mathbf a\in\mathbb Z^2\),
\begin{align}
\mathbf P(\mathbf S=\mathbf a)
=
\int_{\mathbb T^2}
e(-\boldsymbol{\alpha}\cdot\mathbf a)
\mathbf{E} e(\boldsymbol\alpha\cdot \boldsymbol{S})\,
d\boldsymbol{\alpha},
\end{align}
which follows by the two-dimensional version of \cref{eq:fourier-inversion} on $\Z^2$. 

Next, we fix some real numbers $0<a<b$ and an even function $\eta\in C_c^\infty(\mathbb R)$ such that
\begin{align}
    0\leq \eta\leq 1,
    \qquad
    \eta=1 \ \text{on }[-a,a],
    \qquad
    \operatorname{supp}\eta\subset[-b,b].
    \label{eq:smooth-cutoff}
\end{align}
For $R>2b$, define
\begin{align}
    T_R(u)
    &:=
    \int_{-1/2}^{1/2}\eta(R\alpha)e(\alpha u)\,d\alpha.
    \label{eq:smoothing-kernel}
\end{align}
Since $\operatorname{supp}(\eta(R\cdot)) \subset [-b/R,b/R] \subset [-1/2,1/2]$, we may extend the integral to $\mathbb R$. After the change of variables $v=R\alpha$, and using the fact that $\widehat{\eta}$ is even (as $\eta$ is), it follows that
\begin{align}
    T_R(u)=\frac{1}{R}\widehat{\eta}(u/R).
    \label{eq:smoothing-kernel-identity}
\end{align}
We record several useful properties of $T_R$ in the following lemma.

\begin{lemma}
\label{lem:smooth-lattice-kernel}
For every $A,j\geq0$,
\begin{align}
    |T_R^{(j)}(u)|
    &\ll_{A,j}
    R^{-1-j}
    \left(1+\frac{|u|}{R}\right)^{-A}.
    \label{eq:kernel-decay}
\end{align}
Moreover, for every $j\geq0$,
\begin{align}
    \sum_{m\in\mathbb Z}|m|^j|T_R(m)|
    \ll_j R^j.
    \label{eq:kernel-moments}
\end{align}
Finally, $\sum_{m\in\mathbb Z}T_R(m)=1$ and $\sum_{m\in\mathbb Z}mT_R(m)=0$.
\end{lemma}

\begin{proof}
From \cref{eq:smoothing-kernel-identity} we obtain $T_R^{(j)}(u) = R^{-1-j}\widehat\eta^{(j)}(u/R)$. Since $\widehat\eta$ is Schwartz, \eqref{eq:kernel-decay} follows from (\ref{eq:fourier-schwartz-decay}). To prove \eqref{eq:kernel-moments}, fix $j\geq0$ and choose $A>j+1$. Applying \eqref{eq:kernel-decay} with such $A$ gives
\begin{align*}
    \sum_{m\in\mathbb Z}|m|^j|T_R(m)|
    &\ll_{A,j}
    R^{-1}
    \sum_{m\in\mathbb Z}
    |m|^j
    \left(1+\frac{|m|}{R}\right)^{-A} \\
    &=
    R^j
    \left[
        \frac1R
        \sum_{m\in\mathbb Z}
        \left(\frac{|m|}{R}\right)^j
        \left(1+\frac{|m|}{R}\right)^{-A}
    \right]
    \ll_j R^j,
\end{align*}
since $A>j+1$. Furthermore, $T_R$ is the inverse Fourier transform of $\alpha\mapsto\eta(R\alpha)$, and hence $\widehat{T_R}(\xi)=\eta(R\xi)$.
By Poisson summation,
\begin{align*}
    \sum_{m\in\mathbb Z}T_R(m)
    =
    \sum_{k\in\mathbb Z}\eta(Rk)
    =
    \eta(0)
    =
    1.
\end{align*}
Since $\eta$ is even, $T_R$ is even. Hence, $mT_R(m)$ is odd, and its absolute summability follows from \eqref{eq:kernel-moments} with $j=1$. Therefore, $\sum_{m\in\mathbb Z}mT_R(m)=0$.
\end{proof}
Finally, we record some elementary regularity estimates for the Gaussian density introduced in \cref{eq:intro-Gaussian}. Write $\mathbf y:=\mathbf u-(t,t)$. For \(i,j\in\{0,1\}\), direct differentiation gives
\[
\partial_{u_i}G_t(\mathbf u)
=
-\frac{y_i}{t}G_t(\mathbf u),
\]
\[
\partial_{u_i}\partial_{u_j}G_t(\mathbf u)
=
\left(
\frac{y_i y_j}{t^2}
-
\frac{\delta_{ij}}{t}
\right)
G_t(\mathbf u),
\]
and
\[
\partial_tG_t(\mathbf u)
=
\left(
-\frac1t
+
\frac{y_0+y_1}{t}
+
\frac{\|\mathbf y\|^2}{2t^2}
\right)
G_t(\mathbf u),
\]
where \(\delta_{ij}\) denotes the Kronecker delta. Upon writing \(\mathbf y=\sqrt t\,\mathbf v\), and using
\[
\sup_{\mathbf v\in\mathbb R^2}
\|\mathbf v\|^k e^{-\|\mathbf v\|^2/2}
<\infty
\qquad (k\geq 0),
\]
we obtain, for $t \geq 1$,
\begin{align}
\|\nabla G_t\|_\infty+\|\partial_tG_t\|_\infty
\ll t^{-3/2},
\qquad
\|D^2G_t\|_\infty\ll t^{-2},
\label{eq: Gaussian-Smoothing-Prelim}
\end{align}
where
\[
\|\nabla G_t\|_\infty
:=
\sup_{\mathbf u\in\mathbb R^2}
\|\nabla G_t(\mathbf u)\|,
\qquad
\|\partial_tG_t\|_\infty
:=
\sup_{\mathbf u\in\mathbb R^2}
|\partial_tG_t(\mathbf u)|,
\]
and $D^2G_t$ denotes the Hessian of $G_t$ with respect to $\mathbf u$.

We shall also use the following standard Gaussian estimates on \(\mathbb Z^2\), which we make use of in Section \ref{sec:dynamical}. These estimates follow from standard Gaussian integral comparisons and the fundamental theorem of calculus. For \(t\geq 1\),
\begin{align}
\label{eq:gaussian-dynamical-1}
\sum_{\mathbf m\in\mathbb Z^2}G_t(\mathbf m)
=1+O(t^{-1/2}).
\end{align}
Moreover, for \(r\geq 0\), let $Q(t,r):=\left\{\mathbf m\in\mathbb Z^2:|m_i-t|\leq r,\ i=0,1\right\}$. Then,
\begin{align}
\label{eq:gaussian-dynamical-2}
\sum_{\mathbf m\notin Q(t,r)}G_t(\mathbf m)
\ll \exp\left(-\frac{r^2}{4t}\right).
\end{align}
Finally, writing \(\mathbf e_0=(1,0)\) and \(\mathbf e_1=(0,1)\), we have
\begin{align}
\label{eq:gaussian-dynamical-3}
\sum_{\mathbf m\in\mathbb Z^2}
|G_t(\mathbf m-\mathbf e_i)-G_t(\mathbf m)|
\ll t^{-1/2},
\qquad i=0,1.
\end{align}

\subsection{Quantitative Two-point Correlations}\label{subsec:prelim-correlations}

We use the non-pretentious case of the quantitative correlation theorem from Tao and Ter\"av\"ainen~\cite[Theorem~3.1(ii)]{TaoTeravainen2025}, specialised to $W=1$ and the shifts $0,1$.

\begin{theorem}[Tao--Ter\"av\"ainen]\label{thm:prelim-TT}
There exists an absolute $c>0$ with the following property.  Let
$X\geq2$, $1\leq \mathcal{L} \leq\log X$, and let $g_1,g_2:\N\to\C$ be $1$-bounded
multiplicative functions satisfying
\begin{equation}\label{eq:prelim-TT-hypothesis}
 \exp M\bigl(g_1;X^2,(\log X)^{1/125}\bigr)\gg \mathcal{L}.
\end{equation}
There is a measurable set $\mathcal{E}\subset[\sqrt X,X]$ such that
\begin{equation}\label{eq:prelim-TT-exceptional}
 \frac1{\log X}\int_{\mathcal{E}}\frac{\mathrm dN}{N}\ll \mathcal{L}^{-c}
\end{equation}
and, for each $(h_1,h_2)=(0,1)$ or $(h_1,h_2)=(1,0)$, one has
\begin{align}
    \frac1N
    \sum_{N<n\leq2N}
    g_1(n+h_1)g_2(n+h_2)
    \ll \mathcal{L}^{-c}
\end{align}
for every $N\in[\sqrt X,X]\setminus\mathcal E$.
\end{theorem}

For $\nu\in\{\omega,\Omega\}$ and $\alpha\in\T$, set $f_{\nu,\alpha}(n)=e(\alpha\nu(n))$. These functions are multiplicative, $1$-bounded, and satisfy $f_{\nu,\alpha}(p)=e(\alpha)$ for every prime. The non-pretentiousness argument used in the high-frequency analysis of \cite[p. 36]{TaoTeravainen2025} yields the following uniform estimate.

\begin{lemma}\label{lem:nonpret}
For sufficiently large $X$, uniformly in $\nu\in\{\omega,\Omega\}$ and
$\alpha\in\T$,
\begin{equation}\label{eq:prelim-phase-distance}
 M\bigl(f_{\nu,\alpha};X^2,(\log X)^{1/125}\bigr)
 \geq c_0\|\alpha\|_{\T}^2\log_2X-O(1)
\end{equation}
for an absolute $c_0>0$. 
\end{lemma}

\section{The Circle Method and High Frequencies}\label{sec:circle-method}

To prove Theorem \ref{thm:intro-main}, it will suffice to establish the following averaged dyadic statement concerning the normalized joint counts, whose proof is the main objective of this and the following section.

\begin{theorem}\label{thm:intro-dyadic}
There is an absolute constant $c_0>0$ such that, for all sufficiently large $X$,
\begin{equation}\label{eq:intro-dyadic}
 \frac1{\log X}\int_{\sqrt X}^{X}
 \max_{\nu_0,\nu_1\in\{\omega,\Omega\}}
 \sup_{\boldsymbol m\in\Z^2}
 \bigl|\Tilde{P}_x^{\nu_0,\nu_1}(\boldsymbol m)-G_{\log_2x}(\boldsymbol m)\bigr|
 \,\frac{\mathrm dx}{x}
 \ll (\log\log X)^{-1-c_0},
\end{equation}
where, for $\boldsymbol{m} = (m_0, m_1) \in \Z^2$, $\Tilde{P}_x$ is the dyadic normalized count
\begin{align}
 \Tilde{P}_x^{\nu_0,\nu_1}(\boldsymbol m)
 :=\frac1x \cdot \#\{x<n\leq2x:\nu_0(n)=m_0,\ \nu_1(n+1)=m_1\}.
\end{align}
\end{theorem}

\begin{proof}[Deduction of Theorem~\ref{thm:intro-main}]
We denote by $P_x^{\nu_0,\nu_1}(\boldsymbol m)$ the normalized joint count in \cref{eq:joint-local-frequency}. We initially consider $x\in[X^{3/4},X]$, a fixed range of scales. This allows us to decompose $[1,x]$ into dyadic intervals while keeping all relevant dyadic scales $x/2^j$ inside \([\sqrt X,X]\), where Theorem \ref{thm:intro-dyadic} applies. Applying Theorem \ref{thm:intro-dyadic} will yield the conclusion of Theorem \ref{thm:intro-main} on the restricted range $[X^{3/4},X]$, outside of an exceptional set $\mathcal E_X$ of small logarithmic measure. Since $\mathcal E_X$ depends on $X$, it remains to use these scale-dependent exceptional sets to construct one single set $\mathcal X\subset\mathbb N$ of good scales. Note that there is nothing particular about the exponent $3/4$ and any fixed exponent strictly between
$1/2$ and $1$ would suffice. 

To that end, if we take $J=\lceil C\log_3X\rceil$, with $C>0$ being sufficiently large, dyadic decomposition gives
\begin{align}
 P_x^{\nu_0,\nu_1}(\boldsymbol m)
 =\sum_{j=0}^{J-1}2^{-j-1}
 \Tilde{P}_{x/2^{j+1}}^{\nu_0,\nu_1}(\boldsymbol m)+O(2^{-J}).
 \label{eq:dyadic}
\end{align}
Notice that the scales $x/2^{j+1}$ belong to $[\sqrt X,X]$ for all $0\leq j\leq J-1$. Moreover, two elementary calculations yield $|\log_2({x/2^{j+1}})-\log_2x|\ll J/\log x$ and
$\sup_{\boldsymbol m}|\partial_tG_t(\boldsymbol m)|\ll t^{-3/2}$. Applying the mean value theorem, we obtain
\begin{align}
\sup_{\mathbf m\in\mathbb Z^2}
\left|
G_{\log_2({x/2^{j+1}})}(\mathbf m)
-
G_{\log_2x}(\mathbf m)
\right|
\ll (\log_2X)^{-3/2}\frac{J}{\log X}.
\label{eq:gausiaan-dyadic}
\end{align}
Subtracting $G_{\log_2x}(\mathbf m)$ from both sides of \eqref{eq:dyadic} and using \eqref{eq:gausiaan-dyadic}, we obtain
\begin{align*}
P_x^{\nu_0,\nu_1}(\mathbf m)-G_{\log_2 x}(\mathbf m)
&=
\sum_{j=0}^{J-1}
2^{-j-1}
\left(
\widetilde P_{x/2^{j+1}}^{\nu_0,\nu_1}(\mathbf m)
-
G_{\log_2(x/2^{j+1})}(\mathbf m)
\right)
\notag\\
& \quad +
O\!\left(
2^{-J}
+
(\log_2 X)^{-3/2}\frac{J}{\log X}
\right).
\end{align*}
Taking the maximum over $(\nu_0, \nu_1)$ and the supremum over $\boldsymbol{m}$, and integrating against $\mathrm dx/x$ on $[X^{3/4},X]$, we can apply Theorem \ref{thm:intro-dyadic} to obtain
\begin{align*}
    \frac{1}{\log X}
\int_{X^{3/4}}^{X}
\max_{\nu_0,\nu_1\in\{\omega,\Omega\}}
 \sup_{\boldsymbol m\in\Z^2}
 \bigl|P_x^{\nu_0,\nu_1}(\boldsymbol m)-G_{\log_2x}(\boldsymbol m)\bigr|\,\frac{dx}{x}
\ll
(\log_2 X)^{-1-c_0},
\end{align*}
for some constant $c_0 > 0$. Passing from real to integer scales comes at a negligible error, so we also obtain 
\begin{align*}
    \frac{1}{\log X} \sum_{X^{3/4} \leq x \leq X} \frac{1}{x} \max_{\nu_0,\nu_1\in\{\omega,\Omega\}}
 \sup_{\boldsymbol m\in\Z^2}
 \bigl|P_x^{\nu_0,\nu_1}(\boldsymbol m)-G_{\log_2x}(\boldsymbol m)\bigr| \ll
(\log_2 X)^{-1-c_0},
\end{align*}
where $x$ is now an integer.

A standard application of Markov's inequality alongside the estimate $\log_2 x=\log_2 X+O(1)$ on $[X^{3/4},X]$, shows that there exists a set $\mathcal{E}_X\subset [X^{3/4},X]\cap\mathbb N$ such that
\begin{align*}
\frac{1}{\log X}\sum_{x\in \mathcal{E}_X}\frac{1}{x}
\ll
(\log_2 X)^{-c_0/2},
\end{align*}
and
\begin{align*}
P_x^{\nu_0,\nu_1}(\mathbf m)
=
G_{\log_2 x}(\mathbf m)
+
O\!\left((\log_2 x)^{-1-c_0/2}\right)
\end{align*}
uniformly for $x\in ([X^{3/4},X]\cap\N) \setminus \mathcal{E}_X$, $\mathbf m\in\mathbb Z^2$, and $\nu_0,\nu_1\in\{\omega,\Omega\}$.

In order to finalise deduction of Theorem \ref{thm:intro-main} we need one common set $\mathcal{X}$ of good scales. Therefore, fix $0<\alpha<\log(4/3)$ and define $Y_j:=\exp(\exp(\alpha j)), \, j\geq 1$. It follows that $[Y_j,Y_{j+1}) \subseteq [Y_{j+1}^{3/4},Y_{j+1}]$, and so, for $j$ large enough, we can define the sets
\begin{align*}
\mathcal B_j
:=
\mathcal E_{Y_{j+1}}\cap [Y_j,Y_{j+1}),
\end{align*}
and choosing $j_0$ sufficiently large, we can construct the set
\begin{align*}
\mathcal X
:=
\bigcup_{j\geq j_0}
\left(
[Y_j,Y_{j+1})\cap\mathbb N\setminus\mathcal B_j
\right).
\end{align*}
It is now not too difficult to confirm that Theorem \ref{thm:intro-main} holds with constant $c_0/2$.
\end{proof}

\subsection{Two-frequency Decomposition}\label{subsec:circle-decomposition}

We turn to the proof of Theorem~\ref{thm:intro-dyadic}. Adopting the circle-method approach of \cite[Section~4]{TaoTeravainen2025}, we separate the two Fourier frequencies into a neighbourhood of the origin and its complement. In this subsection we make the reduction precise and in the next subsection, we establish the high-frequency estimate.

The setup throughout the rest of this section and Section \ref{sec:low-frequencies} is as follows. We set $X$ to be a sufficiently large asymptotic
parameter tending to infinity, and set
\begin{align*}
    R:=\big(\log_2X\big)^{1/4},
    \qquad
    I_X:=[\sqrt X,X].
\end{align*}
Throughout this and the following section, unless stated otherwise, all estimates are
uniform for $x\in I_X$, $(\nu_0,\nu_1)\in\{\omega,\Omega\}^2$ and $\mathbf{m}\in \Z^2$. In particular, all implied constants are independent of $x,\nu_0,\nu_1 , \boldsymbol{m}$.
For $x\in I_X$,
\begin{align}
    \log_2X-\log2\leq\log_2x\leq\log_2X,
    \label{eq:log-comparison}
\end{align}
and hence $\log_2x=\log_2X+O(1)$ and
$\log_2x\asymp\log_2X$. We shall use this comparability in error terms.

We proceed to apply the circle method with a pair of frequencies $\boldsymbol\alpha=(\alpha_0,\alpha_1)\in\T^2$. Define
\begin{equation}\label{eq:circle-correlation}
 C_x^{\nu_0, \, \nu_1}(\boldsymbol\alpha)
 :=\frac1x\sum_{x<n\leq2x}
 e\bigl(\alpha_0\nu_0(n)+\alpha_1\nu_1(n+1)\bigr).
\end{equation}
Using the standard orthogonality relation for characters on $\mathbb T$, we obtain the exact identity
\begin{align*}
\one_{\{\nu_0(n)=m_0,\ \nu_1(n+1)=m_1\}}
 =\int_{\T^2}
 e\bigl(\alpha_0(\nu_0(n)-m_0)+\alpha_1(\nu_1(n+1)-m_1)\bigr)
 \,\mathrm d\alpha_1\mathrm d\alpha_0.
\end{align*}
Therefore, averaging over $x<n\leq2x$ yields
\begin{equation}\label{eq:circle-inversion}
 \Tilde{P}_x^{\nu_0,\, \nu_1}(\boldsymbol m)
 =\int_{\T^2}e(-\boldsymbol\alpha\cdot\boldsymbol m)
        C_x^{\nu_0, \, \nu_1}(\boldsymbol\alpha)
        \,\mathrm d\boldsymbol\alpha.
\end{equation}
Following (\ref{eq:smooth-cutoff}) with $a=1, b =2$, we set  $\eta\in C_c^\infty(\mathbb R)$ to be an even function such that
\begin{align*}
    0\leq\eta\leq1,
    \qquad
    \eta=1 \text{ on }[-1,1],
    \qquad
    \operatorname{supp}\eta\subset[-2,2].
\end{align*}
For $\boldsymbol{\alpha}=(\alpha_0,\alpha_1)\in\mathbb T^2$, define
\begin{align}
    \Psi_R(\boldsymbol{\alpha})
    :=
    \eta(R\alpha_0)\eta(R\alpha_1).
    \label{eq:two-dimensional-cutoff}
\end{align}
For $X$ sufficiently large, we have $R>4$, thus having the support of $\eta(R\cdot)$ contained strictly inside $[-1/2,1/2)$, and hence $\Psi_R$ defines a smooth function on $\mathbb T^2$. Moreover, $0\leq\Psi_R\leq1$ and
\begin{align*}
    \Psi_R(\boldsymbol{\alpha}) &= 1
    \qquad\text{if}\qquad
    \max_{i=0,1}\|\alpha_i\|_{\mathbb T}\leq \frac{1}{R},\\
    \operatorname{supp}\Psi_R
    &\subset
    \left\{
        \boldsymbol{\alpha}\in\mathbb T^2:
        \max_{i=0,1}\|\alpha_i\|_{\mathbb T}\leq \frac{2}{R}
    \right\}.
\end{align*}
Inserting $1=\Psi_R+(1-\Psi_R)$ into \eqref{eq:circle-inversion}, write
\begin{equation}\label{eq:circle-split}
 \Tilde{P}_x^{\nu_0, \, \nu_1}(\boldsymbol m)
 =L_x^{\nu_0, \, \nu_1}(\boldsymbol m)+H_x^{\nu_0, \, \nu_1}(\boldsymbol m),
\end{equation}
where
\begin{align}
 L_x^{\nu_0, \, \nu_1}(\boldsymbol m)
 &:=\int_{\T^2}\Psi_R(\boldsymbol\alpha)
       e(-\boldsymbol\alpha\cdot\boldsymbol m)
       C_x^{\nu_0, \, \nu_1}(\boldsymbol\alpha)
       \,\mathrm d\boldsymbol\alpha, \label{eq:circle-low-definition}\\
H_x^{\nu_0, \, \nu_1}(\boldsymbol m)
 &:=\int_{\T^2}(1-\Psi_R(\boldsymbol\alpha))
       e(-\boldsymbol\alpha\cdot\boldsymbol m)
       C_x^{\nu_0, \, \nu_1}(\boldsymbol\alpha)
       \,\mathrm d\boldsymbol\alpha. \label{eq:circle-high-definition}
\end{align}
\begin{remark}
    The dependence of $L_x^{\nu_0, \, \nu_1}(\boldsymbol m)$ and $H_x^{\nu_0, \, \nu_1}(\boldsymbol m)$ on the
    ambient parameter $X$, through $R$, is suppressed in the notation. 
\end{remark}
To obtain estimates independent of $\boldsymbol{m}$, set
\begin{align}
 L_x^*&:=\max_{\nu_0,\nu_1\in\{\omega,\Omega\}}
       \sup_{\boldsymbol m\in\Z^2}
       |L_x^{\nu_0, \, \nu_1}(\boldsymbol m)
                         -G_{\log_2x}(\boldsymbol m)|,
       \label{eq:circle-low-majorant}\\
 H_x^*&:=\max_{\nu_0,\nu_1\in\{\omega,\Omega\}}
       \int_{\T^2}(1-\Psi_R(\boldsymbol\alpha))
       |C_x^{\nu_0, \, \nu_1}(\boldsymbol\alpha)|
       \,\mathrm d\boldsymbol\alpha.
       \label{eq:circle-high-majorant}
\end{align}
In order to establish (\ref{eq:intro-dyadic}), it suffices to prove the high-frequency estimate
\begin{align}
    \int_{I_X} H_x^*\,\frac{dx}{x}
    = o( (\log_2X)^{-1-c_1} \log X),
    \label{eq:circle-high-target}
\end{align}
for some absolute constant $c_1>0$, and the low-frequency estimate
\begin{align}
    L_x^*&\ll(\log_2X)^{-1-c_2}
    \label{eq:circle-low-target}
\end{align}
 for some absolute constant $c_2 >0$. Indeed, since $|e(-\boldsymbol{\alpha}\cdot\boldsymbol m)|=1$ and
$0\leq\Psi_R\leq1$, we obtain
\begin{align*}
    \left|H_x^{\nu_0, \, \nu_1}(\boldsymbol m)\right|
    &\leq
    \int_{\mathbb T^2}
    (1-\Psi_R(\boldsymbol{\alpha}))
    \left|C_x^{\nu_0, \, \nu_1}(\boldsymbol{\alpha})\right|
    \,d\boldsymbol{\alpha} \leq H_x^*.
\end{align*}
Consequently, by the triangle inequality
\begin{align*}
    \max_{\nu_0,\nu_1\in\{\omega,\Omega\}}
    \sup_{\boldsymbol m\in\mathbb Z^2}
    \left|
        \Tilde{P}_x^{\nu_0,\nu_1}(\boldsymbol m)
        -
        G_{\log_2 x}(\boldsymbol m)
    \right|
    &\leq L_x^*+H_x^*.
\end{align*}
Assuming (\ref{eq:circle-high-target}) and (\ref{eq:circle-low-target}), and upon setting $c_0:=\min\{c_1,c_2\}>0$, we obtain (\ref{eq:intro-dyadic}), as desired.

\subsection{The High-frequency Estimate}\label{subsec:circle-high}

The proof of the high-frequency estimate (\ref{eq:circle-high-target}) closely follows the approach of \cite{TaoTeravainen2025} as it primarily relies on Theorem \ref{thm:prelim-TT} and the uniform estimate from Lemma \ref{lem:nonpret} for $\|\alpha\|_{\T}\geq R^{-1}$. We record the necessary technical adaptations in the following stronger estimate that immediately implies (\ref{eq:circle-high-target}).

\begin{proposition}\label{prop:circle-high}
There is an absolute constant $\kappa>0$ such that
\begin{align}
 \frac1{\log X}\int_{I_X}H_x^*\,\frac{\mathrm dx}{x}
 \ll\exp(-\kappa\sqrt{\log_2 X})
\end{align}
for all sufficiently large $X$.
\end{proposition}

\begin{proof}
Fix an ordered pair $(\nu_0,\nu_1)\in\{\omega,\Omega\}^2$, and a frequency $\boldsymbol{\alpha}=(\alpha_0,\alpha_1)\in\mathbb T^2$ such that $1-\Psi_R(\boldsymbol{\alpha})\neq 0$. By the defining properties of $\Psi_R$, we have $\Psi_R(\boldsymbol{\alpha})=1$ whenever $\max_{i=0,1}\|\alpha_i\|_{\mathbb T}\leq R^{-1}$. Consequently, there exists an index $i\in\{0,1\}$ such that $ \|\alpha_i\|_{\mathbb T}>R^{-1}$.

We set $f_{\nu,\alpha}(n):=e(\alpha\nu(n))$ as aforementioned. Let
\begin{align*}
    g_1:=f_{\nu_i,\alpha_i},
    \qquad
    h_1:=i,
\end{align*}
and
\begin{align*}
    g_2:=f_{\nu_{1-i},\alpha_{1-i}},
    \qquad
    h_2:=1-i.
\end{align*}
Thus, $h_1\neq h_2$, and
\begin{align*}
    \frac1x
    \sum_{x<n\leq2x}
    g_1(n+h_1)g_2(n+h_2)
    &=
    C_x^{\nu_0,\nu_1}(\boldsymbol{\alpha}).
\end{align*}
By Lemma \ref{lem:nonpret} and $\|\alpha_i\|_{\mathbb T}>R^{-1}$,
\begin{align*}
    M\!\left(
        g_1;X^2,(\log X)^{1/125}
    \right)
    &\geq
    \tilde{c}\|\alpha_i\|_{\mathbb T}^2\log_2X-O(1)\\
    &\geq
    \tilde{c}\sqrt{\log_2X}-O(1).
\end{align*}
Hence, after decreasing the absolute constant if necessary, there
exists an absolute $c_*>0$ such that, for all sufficiently large $X$,
\begin{align}
    M\!\left(
        g_1;X^2,(\log X)^{1/125}
    \right)
    \geq
    c_*\sqrt{\log_2X}.
    \label{eq:pretentious-lower-bound}
\end{align}
We now choose the parameter in Theorem \ref{thm:prelim-TT}. Set
\begin{align*}
    \mathcal L
    :=
    \exp\!\left(\frac{c_*}{2}\sqrt{\log_2X}\right).
\end{align*}
For sufficiently large $X$, we have $1\leq\mathcal L\leq\log X$, since $\big(c_*\sqrt{\log_2X}\big) / 2 \leq \log_2X$.
Moreover, by \eqref{eq:pretentious-lower-bound}, $\exp M\!\left(g_1;X^2,(\log X)^{1/125}\right) \geq \mathcal L^2
    \gg \mathcal L$.
Thus, the non-pretentious hypothesis of Theorem \ref{thm:prelim-TT} is satisfied. Let $c>0$ denote the absolute constant occurring in Theorem \ref{thm:prelim-TT}. Applying the theorem produces a measurable set $\mathcal{E} =\mathcal{E}_{X,(\nu_0,\nu_1),\boldsymbol{\alpha}}\subset I_X$ such that
\begin{align}
    \frac1{\log X}
    \int_{\mathcal{E}}\frac{dx}{x}
    \ll
    \mathcal L^{-c},
    \label{eq:TT-exceptional-measure}
\end{align}
and, for every $x\in I_X\setminus \mathcal{E}$,
\begin{align}
    \left|
        C_x^{\nu_0, \nu_1}(\boldsymbol{\alpha})
    \right|
    \ll
    \mathcal L^{-c}.
    \label{eq:TT-correlation-bound}
\end{align}
Splitting the integral into $\mathcal{E}$ and its complement, and using \eqref{eq:TT-exceptional-measure} and \eqref{eq:TT-correlation-bound}, we obtain
\begin{align*}
    \frac1{\log X}
    \int_{I_X}
    \left|
        C_x^{\nu_0, \nu_1}(\boldsymbol{\alpha})
    \right|
    \frac{dx}{x}
    &\leq
    \frac2{\log X}
    \int_{\mathcal{E}}\frac{dx}{x}+
    \frac1{\log X}
    \int_{I_X\setminus \mathcal{E}}
    \left|
        C_x^{\nu_0, \nu_1}(\boldsymbol{\alpha})
    \right|
    \frac{dx}{x}\\
    &\ll
    \mathcal L^{-c}
\end{align*}
Following the definition of $\mathcal L$, by setting $\kappa := c_*c/2$, we have proved
\begin{align}
    \frac1{\log X}
    \int_{I_X}
    \left|
        C_x^{\nu_0, \nu_1}(\boldsymbol{\alpha})
    \right|
    \frac{dx}{x}
    \ll
    e^{-\kappa\sqrt{\log_2X}}
    \label{eq:uniform-averaged-correlation}
\end{align}
whenever $1-\Psi_R(\boldsymbol{\alpha})\neq0$. The implied constant and $\kappa$ are independent of $(\nu_0, \nu_1)$ and $\boldsymbol{\alpha}$. 
We now return to $H_x^*$. By definition,
\[
    H_x^*=\max_{\nu_0,\nu_1\in\{\omega,\Omega\}}
       \int_{\T^2}(1-\Psi_R(\boldsymbol\alpha))
       |C_x^{\nu_0, \, \nu_1}(\boldsymbol\alpha)|
       \,\mathrm d\boldsymbol\alpha.
\]
Since the integrand is nonnegative, Tonelli's theorem gives
\begin{align*}
    \frac1{\log X}
    \int_{I_X}H_x^*\,\frac{dx}{x}
    &\leq
    \sum_{(\nu_0,\nu_1)\in\{\omega,\Omega\}^2}
    \int_{\mathbb T^2}
    (1-\Psi_R(\boldsymbol{\alpha}))
    \left[
        \frac1{\log X}
        \int_{I_X}
        \left|
            C_x^{\nu_0,\nu_1}(\boldsymbol{\alpha})
        \right|
        \frac{dx}{x}
    \right]
    d\boldsymbol{\alpha}.
\end{align*}
Whenever
$1-\Psi_R(\boldsymbol{\alpha})\neq0$, the inner bracketed term is bounded by
\eqref{eq:uniform-averaged-correlation}. Hence,
\begin{align*}
    \frac1{\log X}
    \int_{I_X}H_x^*\,\frac{dx}{x}
    &\ll
    e^{-\kappa\sqrt{\log_2X}},
\end{align*}
because $0\leq1-\Psi_R\leq1$, and $\mathbb T^2$ has normalised Haar measure one.
\end{proof}

\section{The Low-frequency Estimate}\label{sec:low-frequencies}

We now address the remaining low-frequency estimate \eqref{eq:circle-low-target}. To that end, we expand \eqref{eq:circle-low-definition}, interchange the finite sum and the integral, and apply the one-dimensional identity from \eqref{eq:smoothing-kernel-identity} in each coordinate, to obtain
\begin{align}
 L_x^{\nu_0, \nu_1}(\boldsymbol m)
 &=\frac1{xR^2}\sum_{x<n\leq2x}
   \widehat\eta\!\left(\frac{\nu_0(n)-m_0}{R}\right)
   \widehat\eta\!\left(\frac{\nu_1(n+1)-m_1}{R}\right)
       \notag\\
       &=\frac1x\sum_{x<n\leq2x}
 T_R\bigl(\nu_{0}(n)-m_0\bigr)
 T_R\bigl(\nu_{1}(n + 1)-m_1\bigr).
 \label{eq:circle-low-kernel}
\end{align}
Ultimately, we establish that for $x\in I_X$, $\boldsymbol m\in\mathbb Z^2$, and
$\nu_0,\nu_1\in\{\omega,\Omega\}$,
\begin{align}
    L_x^{\nu_0,\nu_1}(\boldsymbol m)
    =
    G_{\log_2 x}(\boldsymbol m)
    +
    O\!\left((\log_2 X)^{-9/8}\right),
    \label{eq:low-frequency-final}
\end{align}
which immediately implies \eqref{eq:circle-low-target} with $c_2=1/8$.

Following the argument of \cite[Section 4.2]{TaoTeravainen2025}, we first introduce a truncation of the prime-factor counts at a suitable scale \(z\) and use a sieve-theoretic argument to transfer the resulting arithmetic average to a suitable probabilistic model. The main new input is a uniform bivariate local limit theorem for this model, valid simultaneously for all four ordered pairs in $\{\omega,\Omega\}^2$. In addition to providing the Gaussian approximation, the local theorem yields the concentration estimate needed to justify the aforementioned truncation at the two-dimensional lattice scale. Finally, Lemma \ref{lem:smooth-lattice-kernel} allows us to remove the smoothing and recover $G_{\log_2 x}(\boldsymbol m)$ with a power-saving error, thereby establishing \eqref{eq:circle-low-target}.

\subsection{Truncation and the Sieve-theoretic Reduction}\label{subsec:low-sieve-transfer}

\subsubsection{Choice of Scales and Truncation} \label{subsubsec:low-truncation}

The setup remains the same as in the previous section, with two additional
scales $z:=x^{R^{-1/2}}$ and $H:=x^{R^{-1/4}}$, where \(x\in I_X=[\sqrt X,X]\). Thus $z$ and $H$ depend on \(x\) and \(X\), whereas \(R\) depends only on \(X\). In particular,
\begin{align*}
\frac{\log H}{\log z}=R^{1/4},
\qquad
\log_2 z=\log_2 x-\frac12\log R \asymp \log_2 X.
\end{align*}
The scale $z$ serves as a cutoff for the small prime factors of $n$ and $n+1$, allowing us to work with their truncated prime-divisor counts. The scale $H$ is used to restrict the corresponding $z$-smooth components to a range where the sieve-theoretic reduction can be carried out uniformly.
Recall that $n_{\le z}$ denotes the entire \(z\)-smooth component of $n$, including prime-power multiplicities, and that $\nu_{\le z}(n):=\nu(n_{\le z}), \,
\nu\in\{\omega,\Omega\}$. For $\mathbf{m}=(m_0,m_1)\in\mathbb Z^2$, define
\begin{align}
\widetilde L_x^{\nu_0,\nu_1}(\mathbf{m})
:=
\frac1x\sum_{x<n\le2x}
T_R\bigl(\nu_0(n_{\le z})-m_0\bigr)
T_R\bigl(\nu_1((n+1)_{\le z})-m_1\bigr).
\label{eq:low-truncated-kernel}
\end{align}
Our eventual reduction, as in the low-frequency argument of \cite[Section~4.2]{TaoTeravainen2025}, is to replace $L_x^{\nu_0,\nu_1}(\mathbf{m})$ by $\widetilde L_x^{\nu_0,\nu_1}(\mathbf m)$, that is, to replace the full prime-factor counts by their small-prime parts. Before doing so, we emphasise the accuracy required of this replacement.

Consider the scale $\log_2 X$. We know from \cref{eq:log-comparison} that $\log_2x \asymp \log_2X$, and hence the
Gaussian main term $G_{\log_2 x}(\mathbf m)$ appearing in the target estimate \eqref{eq:intro-dyadic} satisfies
\[
\sup_{\mathbf m\in\mathbb Z^2} G_{\log_2 x}(\mathbf m)
\asymp (\log_2 X)^{-1}.
\]
More precisely, for any fixed $C>0$, in the central range $|m_i-\log_2 x| \leq C \sqrt{\log_2X}$, one has $G_{\log_2 x}(\mathbf m)\asymp_C (\log_2 X)^{-1}$. Thus, in order for the replacement of \(L_x^{\nu_0,\nu_1}(\mathbf m)\) by \(\widetilde L_x^{\nu_0,\nu_1}(\mathbf m)\) to be negligible relative to the Gaussian term, the truncation error must be of order $o(({\log_2 X})^{-1})$. For the power-saving estimate required in \eqref{eq:circle-low-target}, we in fact need an error of size \(O(({\log_2 X})^{-1-c})\) for some \(c>0\). This is in natural contrast to the one-frequency argument of \cite{TaoTeravainen2025}, where the relevant local Gaussian term is of order $({\log_2 X})^{-1/2}$. 

Proceeding with a crude truncation estimate is insufficient for our purposes at the present choice of scales. We now isolate precisely where this loss appears. For
\(i\in\{0,1\}\), write
\begin{align}
u_i(n):=\nu_i((n+i)_{\le z})-m_i,
\qquad
\Delta_i(n):=\nu_i(n+i)-\nu_i((n+i)_{\le z}).
\label{eq:truncation-quantities-section4}
\end{align}
Therefore, $\nu_i(n+i)-m_i=u_i(n)+\Delta_i(n)$. For either $\nu_i=\omega$ or $\nu_i=\Omega$, the contribution \(\Delta_i(n)\) comes only from primes exceeding \(z\), and their product divides $n + i$. Consequently,
\begin{align*}
z^{\Delta_i(n)}
\leq (n+i)_{>z}
\leq 2x+1,
\end{align*}
and thus, $0\leq \Delta_i(n) \ll \sqrt R$
by the definition of $z$. Motivated by \eqref{eq:kernel-decay}, fix a sufficiently large integer \(A\) (say \(A=100\)), and define $W_R:\R\to \R$ by
\begin{align}
W_R(u):=
\left(1+\frac{|u|}{R}\right)^{-A}.
\end{align}
If $|v|\ll\sqrt R$ and $0\le\theta\le1$, then
\[
1+\frac{|u+\theta v|}{R}
\asymp
1+\frac{|u|}{R}
\]
uniformly in \(u\). Hence, Lemma \ref{lem:smooth-lattice-kernel} gives
\begin{align}
&|T_R'(u+\theta v)|
\ll R^{-2}W_R(u),\\
&|T_R(u+v)|+|T_R(u)|
\ll R^{-1}W_R(u).
\end{align}
By the fundamental theorem of calculus,
\begin{align}
|T_R(u+v)-T_R(u)| \ll R^{-3/2}W_R(u).
\end{align}
Combining these estimates and telescoping the
product, we obtain
\[
\begin{aligned}
\bigl|
T_R(u_0+v_0)T_R(u_1+v_1)
-
T_R(u_0)T_R(u_1)
\bigr|
&\le
|T_R(u_0+v_0)-T_R(u_0)|\,|T_R(u_1+v_1)| \\
&\quad+
|T_R(u_0)|\,|T_R(u_1+v_1)-T_R(u_1)|\\
&\ll
R^{-5/2}W_R(u_0)W_R(u_1).
\end{aligned}
\]
Consequently,
\begin{align}
\bigl|
L_x^{\nu_0,\nu_1}(\mathbf m)
-
\widetilde L_x^{\nu_0,\nu_1}(\mathbf m)
\bigr|
\ll
R^{-5/2}
\frac1x\sum_{x<n\le2x}
W_R(u_0(n))W_R(u_1(n)).
\label{eq:naive-truncation}
\end{align}
At this point, the naive argument simply uses $W_R\le1$, giving
\begin{align*}
\bigl|
L_x^{\nu_0,\nu_1}(m)
-
\widetilde L_x^{\nu_0,\nu_1}(m)
\bigr|
\ll
R^{-5/2}
=
(\log_2 X)^{-5/8}.
\end{align*}
This does not give the required \(o((\log_2 X)^{-1})\) error, let alone the power saving needed for the low-frequency estimate. One could try to increase $R$ in order to make this crude estimate sufficiently small. If $R=(\log_2 X)^\theta$, then $R^{-5/2}=o((\log_2 X)^{-1})$ whenever $\theta > 2/5$. Such a change, however, requires corresponding changes throughout the argument, including the definitions of \(z\) and \(H\), the high-frequency non-pretentiousness estimate, and the final Gaussian smoothing error. We instead retain the scale chosen above and improve the averaging step. More precisely, in similar fashion to \cite{TaoTeravainen2025}, rather than replacing the two weights by $1$ in \eqref{eq:naive-truncation}, we shall retain them until after the sieve reduction and the bivariate local limit theorem have been established. The concentration estimate obtained there
will provide the required error in \eqref{eq:low-frequency-final}. For now, we work with the truncated version $\widetilde L_x^{\nu_0,\nu_1}(\mathbf{m})$, and justify the validity of this truncation in Proposition \ref{prop:truncation-main}.

\subsubsection{The Sieve-theoretic Reduction}\label{subsubsec:low-exact-sieve}

Recall that we use the scale $H$ to bound the smooth components of $n$ and $n+1$. Essentially, we are able to discard $n \in (x, 2x]$ for which $(n+i)_{\leq z}>H$ at an exponentially small cost, where $i\in\{0,1\}$. We record this approximation without proof in the following lemma, which is due to an elementary argument using Rankin's trick.

\begin{lemma}\label{lem:low-smooth-tail}
There is an absolute constant $c>0$ such that, for $x\in I_X$ and $i\in\{0,1\}$,
\begin{equation}\label{eq:low-arithmetic-tail}
 \frac1x\#\{x<n\leq 2x:(n+i)_{\leq z}>H\}
 \ll\exp(-cR^{1/4}).
\end{equation}
\end{lemma}
\begin{remark}
    The estimate provided by Lemma \ref{lem:low-smooth-tail} was also used in \cite[Section~4.2]{TaoTeravainen2025}. See \cite{Tenenbaum2015} for a more general background on Rankin's method and smooth numbers.
\end{remark}
In light of the result in Lemma \ref{lem:low-smooth-tail}, we make the first important reduction in \eqref{eq:low-truncated-kernel}. Namely, for $d_0,d_1\in\mathbb N_{\leq z}$, define the smooth-part count
\begin{equation}\label{eq:low-exact-count}
 \mathcal N_x(d_0,d_1)
 :=\#\{x<n\leq2x:n_{\leq z}=d_0,\ (n+1)_{\leq z}=d_1\}.
\end{equation}
Notice that $\mathcal N_x(d_0,d_1)=0$ unless $(d_0,d_1)=1$. Consequently, we set 
\begin{equation}\label{eq:low-smooth-pairs}
 \mathcal D_z(H):=
 \{(d_0,d_1)\in\N_{\leq z}^2:(d_0,d_1)=1,\ d_0,d_1\leq H\}.
\end{equation}
We now group the integers $n$ according to their two smooth
components. Set
\begin{align*}
    \mathcal G
    :=
    \left\{
        x<n\leq 2x:
        n_{\leq z}\leq H,\;
        (n+1)_{\leq z}\leq H
    \right\}
\end{align*}
and let $\mathcal{G}^c$ denote its complement. 
Hence, Lemma \ref{lem:low-smooth-tail} gives $\# \mathcal{G}^c/x \ll e^{-cR^{1/4}}$ for some $c > 0$. Now, let $F:\mathbb Z^2\to\mathbb C$ be bounded.  We obtain
\begin{align*}
    \frac1x\sum_{x<n\leq 2x}
    F\bigl(
        \nu_0(n_{\leq z}),
        \nu_{1}((n+1)_{\leq z})
    \bigr)
    &=
    \frac1x\sum_{n\in\mathcal G}
    F\bigl(
        \nu_0(n_{\leq z}),
        \nu_{1}((n+1)_{\leq z})
    \bigr)\\
    & \quad +
    O\left(
        \|F\|_\infty e^{-cR^{1/4}}
    \right).
\end{align*}
For each $n\in\mathcal G$, let $d_0:=n_{\leq z}$ and $d_1:=(n+1)_{\leq z}$.
Then, $d_0,d_1\in\mathbb N_{\leq z}$ and, by the definition of $\mathcal G$, $d_0,d_1\leq H$.
Moreover, $ d_0\mid n$ and $d_1\mid n+1$, which implies that $d_0$ and $d_1$ are coprime.  Thus, $(d_0,d_1)\in\mathcal D_z(H)$. The sets
\begin{align*}
    \left\{
        x<n\leq 2x:
        n_{\leq z}=d_0,\;
        (n+1)_{\leq z}=d_1
    \right\},
    \qquad
    (d_0,d_1)\in\mathcal D_z(H),
\end{align*}
form a disjoint partition of $\mathcal G$.  Therefore,
\begin{align*}
    \sum_{n\in\mathcal G}
    F\bigl(
        \nu_0(n_{\leq z}),
        \nu_{1}((n+1)_{\leq z})
    \bigr)
    =
    \sum_{(d_0,d_1)\in\mathcal D_z(H)}
    F\bigl(
        \nu_0(d_0),
        \nu_1(d_1)
    \bigr)
    \mathcal N_x(d_0,d_1).
\end{align*}
Combining the preceding identities gives
\begin{align}
    \frac1x\sum_{x<n\leq 2x}
    F\bigl(
        \nu_0(n_{\leq z}),
        \nu_{1}((n+1)_{\leq z})
    \bigr)
    &=
    \frac1x
    \sum_{(d_0,d_1)\in\mathcal D_z(H)}
    F\bigl(
        \nu_0(d_0),
        \nu_1(d_1)
    \bigr)
    \mathcal N_x(d_0,d_1)
    \notag\\
    &\quad+
    O\left(
        \|F\|_\infty e^{-cR^{1/4}}
    \right).
    \label{eq:low-grouping}
\end{align}
We further reduce \eqref{eq:low-grouping} by obtaining a uniform asymptotic formula for $\mathcal N_x(d_0,d_1)$, with $(d_0,d_1)\in\mathcal D_z(H)$. 
To that end, for coprime $d_0,d_1\in\N_{\leq z}$, define
\begin{equation}\label{eq:low-pi}
 \pi_z(d_0,d_1):=
 \frac{\varphi(d_0d_1)}{d_0^2d_1^2}
 \prod_{\substack{p\leq z\\p\nmid d_0d_1}}
       \left(1-\frac2p\right),
\end{equation}
and set $\pi_z(d_0,d_1)=0$ when $(d_0,d_1)>1$. The weight $\pi_z(d_0,d_1)$ gives the expected density of the simultaneous conditions $n_{\leq z}=d_0, \, (n+1)_{\leq z}=d_1$. The following lemma makes this precise over $(d_0,d_1)\in\mathcal D_z(H)$. Its proof is a straightforward application of the fundamental lemma of sieve theory, following the sieve-theoretic argument of \cite[Section 4.2]{TaoTeravainen2025}. We provide a short proof for the sake of completeness. 

\begin{lemma}
\label{lem:exact-valuation-sieve}
There exists an absolute constant $c>0$ such that, for all
sufficiently large $X$, uniformly for $x\in I_X$ and
$(d_0,d_1)\in\mathcal D_z(H)$,
\begin{align}
    \mathcal N_x(d_0,d_1)
    &=
    x\,\pi_z(d_0,d_1)
    \left(1+O\!\left(e^{-c\sqrt R}\right)\right)
    +O\!\left(h_0D\log(2D)\right),
    \label{eq:exact-valuation-sieve}
\end{align}
where $h_0:=\varphi\!\left(\operatorname{rad}(d_0d_1)\right)$ and $D:=x^{1/3}$.
\end{lemma}
\begin{proof}
Fix $(d_0,d_1)\in\mathcal D_z(H)$ and let $q_0:=d_0d_1\operatorname{rad}(d_0d_1)$. For $X$  sufficiently large, $z\geq2$, and so $n_{\leq z}$ has the same parity as $n$. Hence, $\mathcal N_x(d_0,d_1)=0$ unless exactly one of $d_0,d_1$ is even. Therefore, we may assume that exactly one of $d_0,d_1$ is even.

The imposed valuations at primes dividing $d_0d_1$
determine exactly $h_0$ admissible residue classes
modulo $q_0$. In each such class, the remaining
conditions amount to sieving the primes
$p\le z$ with $p\nmid d_0d_1$, with local density
$g(p)=2/p$. The required condition in \eqref{eq:sieve-dimension} is satisfied with $\kappa = 2$ and some absolute $K \geq 1$, by an application of Mertens' Theorem.

Applying the fundamental lemma of sieve theory with $D$ and $s = \log D / \log z$,  and using $\sum_{r\le D}\mu^2(r)2^{\omega(r)}\ll D\log(2D)$, gives
\begin{align*}
\mathcal N_x(d_0,d_1)
=x\frac{h_0}{q_0}
\prod_{\substack{p\le z\\p\nmid d_0d_1}}
\left(1-\frac2p\right)
\left(1+O(e^{-c\sqrt R})\right)
+O(h_0D\log(2D)).
\end{align*}
Since
\begin{align*}
\frac{h_0}{q_0}
=\frac{\varphi(d_0d_1)}{d_0^2d_1^2},
\end{align*}
the assertion follows from the definition of
$\pi_z(d_0,d_1)$.
\end{proof}

Combining \eqref{eq:low-grouping} with Lemma~\ref{lem:exact-valuation-sieve} gives the form of the sieve reduction that will be used in the next subsection. Indeed, let
$F:\mathbb Z^2\to\mathbb C$ be bounded. We obtain
\begin{align*}
\frac1x\sum_{x<n\le 2x}
F\bigl(
        \nu_0(n_{\leq z}),
        \nu_{1}((n+1)_{\leq z})
    \bigr)
&=
\sum_{(d_0,d_1)\in\mathcal D_z(H)}
F\!\left(\nu_0(d_0),\nu_1(d_1)\right)\pi_z(d_0,d_1)\\
&\qquad\quad
+O\!\left(
\|F\|_\infty e^{-c\sqrt R}
\sum_{(d_0,d_1)\in\mathcal D_z(H)}
\pi_z(d_0,d_1)
\right)\\
&\qquad\quad
+O\!\left(\|F\|_\infty H^2x^{-1/2}\right)
+O\!\left(\|F\|_\infty e^{-cR^{1/4}}\right).
\end{align*}
In the next subsection we shall identify $\pi_z$ as a probability mass function. Hence, the first error term is actually $O\!\left(\|F\|_\infty e^{-c\sqrt R}\right)$. The second error term follows from the elementary estimate $h_0D\log(2D)\ll x^{1/2}$, together with the trivial bound $ \#\mathcal D_z(H)\le H^2$. Moreover, $H^2x^{-1/2} = x^{-1/2+2R^{-1/4}} \leq x^{-1/4}$ for all sufficiently large $X$, and this is negligible compared with
$e^{-cR^{1/4}}$, for $x\in I_X$. Consequently, after decreasing the absolute constant if necessary,
\begin{align}
    \frac1x\sum_{x<n\le 2x}
F\bigl(
        \nu_0(n_{\leq z}),
        \nu_{1}((n+1)_{\leq z})
    \bigr)
=
&\sum_{(d_0,d_1)\in\mathcal D_z(H)}
F\!\left(\nu_0(d_0),\nu_1(d_1)\right)\pi_z(d_0,d_1)
\notag\\
&+O\!\left(\|F\|_\infty e^{-cR^{1/4}}\right).
\label{eq:low-grouping-two}
\end{align}

\subsubsection{The Probabilistic Interpretation}
\label{subsubsec:low-model}

Up to an exponentially small error, \eqref{eq:low-grouping-two} shows that the distribution of the two $z$-smooth components is governed by the weights $\pi_z(d_0,d_1)$. We now interpret these weights probabilistically, and recognise them as the joint law of a random model, introduced in \cite[Definition 1.5]{TaoTeravainen2025}.

For every $p\leq z$, let $(A_{0,p},A_{1,p})$ be an $\N_0^2$-valued
random vector, with these vectors independent for distinct primes, and
with local law
\begin{align*}
 \mathbf P\bigl((A_{0,p},A_{1,p})=(0,0)\bigr)&=1-\frac2p,
       \\
 \mathbf P\bigl((A_{0,p},A_{1,p})=(j,0)\bigr)
 &=\mathbf P\bigl((A_{0,p},A_{1,p})=(0,j)\bigr)
  =\left(1-\frac1p\right)p^{-j}\quad(j\geq1).
\end{align*}
For each prime $p$, the above law is the limiting joint distribution of the \(p\)-adic valuation pair $(v_p(n),v_p(n+1))$.

Next, define the random smooth integers
\begin{align*}
 D_i:=\prod_{p\leq z}p^{A_{i,p}},\qquad i \in \{0, 1\}.
\end{align*}
They are \(z\)-smooth positive integers almost surely, and satisfy $(D_0,D_1)=1$. For every coprime $d_0,d_1\in\N_{\leq z}$, independence across primes and the law of $(A_{0,p},A_{1,p})$  give
\begin{align}
 \mathbf P(D_0=d_0,D_1=d_1)
 &=\prod_{p^a\parallel d_0d_1}\left(1-\frac1p\right)p^{-a}
    \prod_{\substack{p\leq z\\p\nmid d_0d_1}}\left(1-\frac2p\right)
 \notag\\*
 &=\pi_z(d_0,d_1).
 \label{eq:low-model-weight}
\end{align}
The same identity holds for noncoprime pairs, with both sides zero.
In particular,
\begin{equation}\label{eq:low-model-total-mass}
 \sum_{d_0,d_1\in\N_{\leq z}}\pi_z(d_0,d_1)=1.
\end{equation}
For $j\in\mathbb N_0$, define
\begin{align}
q_\omega(j):=\mathbf 1_{\{j\ge 1\}},
\qquad
q_\Omega(j):=j.
\end{align}
Given $(\nu_0,\nu_1)\in \{\omega, \Omega\}^2$, define
\begin{equation}
\mathbf{Y}_p^{(\nu_0,\nu_1)}
:=
\bigl(q_{\nu_0}(A_{0,p}),q_{\nu_1}(A_{1,p})\bigr),
\qquad
\mathbf{S}^{(\nu_0,\nu_1)}
:=
(S_0^{(\nu_0,\nu_1)},S_1^{(\nu_0,\nu_1)})
:=
\sum_{p\le z}\mathbf{Y}_p^{(\nu_0,\nu_1)},
\end{equation}
where the summation is coordinate-wise. By construction,
\begin{align}
S_i^{(\nu_0,\nu_1)}
=
\sum_{p\le z}q_{\nu_i}(A_{i,p})
=
\nu_i(D_i),
\qquad i \in \{0, 1\}.
\end{align}
Therefore, the same collection of random vectors $(A_{0,p},A_{1,p})$ gives rise to all four models, with the choice of $(\nu_0,\nu_1)$ affecting only the
coordinate maps $q_{\nu_0}$ and $q_{\nu_1}$. When $(\nu_0,\nu_1)$ is fixed, we suppress it from the notation and simply write $\mathbf{Y}_p$ and $\mathbf{S}=(S_0,S_1)$.

\begin{proposition}\label{prop:low-sieve-transfer}
There is an absolute constant $c>0$ such that for $x\in I_X$, $(\nu_0,\nu_1)\in \{\omega, \Omega\}^2$, and every bounded
$F:\Z^2\to\C$,
\begin{equation}\label{eq:low-sieve-transfer}
 \frac1x\sum_{x<n\leq2x}
 F\bigl(
        \nu_0(n_{\leq z}),
        \nu_{1}((n+1)_{\leq z})
    \bigr)
 =\mathbf{E} F(\mathbf{S}^{(\nu_0,\nu_1)})
  +O\bigl(\|F\|_\infty e^{-cR^{1/4}}\bigr).
\end{equation}
\end{proposition}
\begin{remark}
    The implied constant is independent of $F$. In particular, $F$ may depend on $X$, $x$, and the prescribed point $\mathbf m\in\Z^2$.
\end{remark}
\begin{proof}
Recalling \eqref{eq:low-grouping-two},
\begin{align*}
    \frac1x\sum_{x<n\le 2x}
F\bigl(
        \nu_0(n_{\leq z}),
        \nu_{1}((n+1)_{\leq z})
    \bigr)
=
&\sum_{(d_0,d_1)\in\mathcal D_z(H)}
F\!\left(\nu_0(d_0),\nu_1(d_1)\right)\pi_z(d_0,d_1)
\notag\\
&+O\!\left(\|F\|_\infty e^{-cR^{1/4}}\right),
\end{align*} 
and using \eqref{eq:low-model-weight}, the first-term on the right-hand side of \eqref{eq:low-grouping-two} is
\[
 \mathbf E\bigl[F(\mathbf{S}^{(\nu_0,\nu_1)})\one_{\{D_0\leq H,\ D_1\leq H\}}\bigr],
\]
as for the random pair $(D_0,D_1)$, belonging to $\mathcal D_z(H)$ is equivalent simply to $D_0\leq H$ and $D_1\leq H$. It remains to remove this restriction in the model at an acceptable error. We first record the marginal law of each coordinate. For every $j\in\mathbb N_0$,
\begin{align*}
\mathbf P(A_{i,p}=j)
=
\left(1-\frac1p\right)p^{-j}.
\end{align*}
Let $\sigma:=\frac1{\log z}$.  
Independence of the vectors $(A_{0,p},A_{1,p})$ across distinct primes
implies
\begin{align*}
\mathbf{E} D_i^\sigma
&=
\prod_{p\le z}\mathbf{E}\,p^{\sigma A_{i,p}} \\
&=
\prod_{p\le z}
\left(1-\frac1p\right)
\sum_{j\ge0}p^{j(\sigma-1)} \\
&=
\prod_{p\le z}
\frac{1-p^{-1}}{1-p^{\sigma-1}}
\ll 1,
\end{align*}
where the last estimate follows from \eqref{euler-product}. Following the application of Rankin's inequality \eqref{eq:prelim-rankin}, for $i=0,1$,
\begin{align*}
\mathbf P(D_i>H)
\le
H^{-\sigma}\mathbf E D_i^\sigma
\ll H^{-\sigma}=e^{-R^{1/4}}.
\end{align*}
By the union bound,
\begin{align*}
\left|
\mathbf{E}\!\left[
F(\mathbf{S}^{(\nu_0, \nu_1)})
\mathbf 1_{\{D_0>H\text{ or }D_1>H\}}
\right]
\right|
&\ll
\|F\|_\infty e^{-R^{1/4}}.
\end{align*}
Thus, the restriction $D_0,D_1\le H$ may be removed at an exponentially small cost.
\end{proof}

\subsection{A Bivariate Local Limit Theorem}\label{subsec:low-local-limit}
According to Proposition \ref{prop:low-sieve-transfer}, in order to finish the proof of the truncated version of the low-frequency estimate, it remains to understand the law of the random vector $\mathbf{S}^{(\nu_0, \nu_1)}$ introduced in the previous section. To that end, we prove a bivariate local limit theorem for $\mathbf{S}^{(\nu_0, \nu_1)}$, which is the direct two-frequency counterpart of the local limit theorem in \cite[Lemma 4.1]{TaoTeravainen2025}.

The rest of this section is organised as follows. We first calculate the moments of the four models.  We then show that it suffices to establish the local limit theorem for $(\omega,\omega)$, from which the remaining three cases can be deduced. Finally, for the $(\omega,\omega)$ model we give a direct Fourier-type proof of the required local limit theorem.

\subsubsection{Calculation of Moments}
\label{subsubsec:llt-moments}
Recall
\begin{equation}
    \mathbb P(A_{i,p}=j)
    =
    \left(1-\frac1p\right)p^{-j},
    \qquad i\in\{0,1\}, \, j\in\mathbb N_0.
    \label{distribution}
\end{equation}
For $\nu\in\{\omega,\Omega\}$,  $\mathbf{E} q_\nu(A_{i,p})$ and ${\operatorname{\mathbf{Var}}}(q_\nu(A_{i,p}))$ are independent of $i$.

Since $q_\omega(A_{i,p})=\mathbf 1_{\{A_{i,p}\ge1\}}$ is a Bernoulli random variable with parameter $1/p$,
\[
\mathbf{E} q_\omega(A_{i,p})=\frac1p,
\qquad
\operatorname{\mathbf{Var}}(q_\omega(A_{i,p}))=\frac1p-\frac1{p^2}.
\]
For $\Omega$, we have $q_\Omega(A_{i,p}) = A_{i,p}$, and the geometric distribution in \eqref{distribution} gives
\begin{align*}
\mathbf{E} q_\Omega(A_{i,p})=\frac1{p-1},
\quad
\operatorname{\mathbf{Var}}(q_\Omega(A_{i,p}))=\frac{p}{(p-1)^2}.
\end{align*}
In particular, uniformly for $\nu\in\{\omega,\Omega\}$,
\begin{equation}
    \mathbf{E}q_\nu(A_{i,p})=\frac1p+O(p^{-2}), \quad \operatorname{\mathbf{Var}}(q_\nu(A_{i,p})) = \frac1p+O(p^{-2}).
\end{equation}
Now, fix $(\nu_0,\nu_1)\in\{\omega,\Omega\}^2$. Since by construction $A_{0,p}$ and $A_{1,p}$ cannot be both simultaneously positive, it follows that $q_{\nu_0}(A_{0,p})q_{\nu_1}(A_{1,p})=0$ almost surely, and therefore,
\begin{equation}
\operatorname{\mathbf{Cov}}
\bigl(
q_{\nu_0}(A_{0,p}),
q_{\nu_1}(A_{1,p})
\bigr)
=
-\mathbf{E} q_{\nu_0}(A_{0,p})\mathbf{E} q_{\nu_1}(A_{1,p})
=
O(p^{-2}).
\end{equation}
Since the vectors $\mathbf{Y}_p^{(\nu_0,\nu_1)}$ are independent for distinct primes, the mean vector and covariance matrix of $\mathbf{S}^{(\nu_0,\nu_1)}$ are
\begin{align*}
\mu_{(\nu_0,\nu_1)}
:=
\mathbf{E}\mathbf{S}^{(\nu_0,\nu_1)}
&=
\sum_{p\le z}
\bigl(\mathbf{E} q_{\nu_0}(A_{i,p}), \mathbf{E} q_{\nu_1}(A_{i,p})\bigr)
=
(\log_2 z,\log_2 z)+O(1),\\
\Sigma_{(\nu_0,\nu_1)}
:=
\operatorname{\mathbf{Cov}}(\mathbf{S}^{(\nu_0,\nu_1)})
&=
\sum_{p\le z}
\begin{pmatrix}
\operatorname{\mathbf{Var}}(q_{\nu_0}(A_{i,p}))
&
-\mathbf{E} q_{\nu_0}(A_{i,p})\mathbf{E} q_{\nu_1}(A_{i,p})
\\
-\mathbf{E} q_{\nu_0}(A_{i,p})\mathbf{E}q_{\nu_1}(A_{i,p})
&
\operatorname{\mathbf{Var}}(q_{\nu_1}(A_{i,p}))
\end{pmatrix}
\notag\\
&=
\log_2 z \cdot I_2+O(1).
\end{align*}
where $I_2$ denotes the $2 \times 2$ identity matrix. Here we used Mertens' theorem and the convergence of
$\sum_p p^{-2}$. In the first identity, the summation and the error term are understood coordinate-wise, whereas in the second identity, the summation and the error term are understood entry-wise.

We shall also need uniform centered-moment bounds. Writing $\mathbf{Z}_p^{(\nu_0,\nu_1)} := \mathbf{Y}_p^{(\nu_0,\nu_1)} - \mathbf{E} \mathbf{Y}_p^{(\nu_0,\nu_1)}$, we obtain
\begin{equation}
\mathbf{E}\|\mathbf{Z}_p^{(\nu_0,\nu_1)}\|^2\ll \frac1p,
\qquad
\mathbf{E}\|\mathbf{Z}_p^{(\nu_0,\nu_1)}\|^3\ll \frac1p,
\qquad
\sum_{p\le z}
\mathbf{E}\|\mathbf{Z}_p^{(\nu_0,\nu_1)}\|^3
\ll \log_2 z.
\label{eq:centered-moments}
\end{equation}

\subsubsection{Reduction to the \texorpdfstring{$(\omega,\omega)$}{omega--omega} Model}
\label{subsubsec:llt-reduction}

The four pairs $(\omega, \omega), (\omega, \Omega), \, (\Omega, \omega)$ and $(\Omega, \Omega)$ are coupled through the same valuation vectors $(A_{0,p},A_{1,p})$. The identity
\begin{align}
q_\Omega(j)
=
q_\omega(j)+(j-1)_+,
\qquad
(j-1)_+:=\max\{j-1,0\},
\end{align}
shows that $q_\Omega(j)=q_\omega(j)$ for $j=0,1$, and that the two maps differ only when $j\ge2$. Thus the four models differ only on the events \(A_{i,p}\ge2\), corresponding arithmetically to divisibility by $p^2$.

At the level of weak convergence, this correction is negligible after the natural normalization, and one could use L\'evy's continuity theorem \cite[Theorem 4.3]{kallenberg2002} to deduce the same qualitative Gaussian limit for all four pairs from the $(\omega,\omega)$ case. For the local limit theorem (Theorem \ref{thm:2d-llt}), however, we require a uniform quantitative approximation to the probabilities
\[
\mathbb P\!\left(\mathbf{S}^{(\nu_0,\nu_1)}= \mathbf{a}\right),
\qquad \mathbf{a}\in\mathbb Z^2.
\]
We therefore compare each model $\mathbf{S}^{(\nu_0,\nu_1)}$ with the $(\omega,\omega)$ reference model at the level of characteristic functions. For notational clarity, we write
\[
\mathbf{B}_p:=\mathbf{Y}_p^{(\omega,\omega)},
\qquad
\mathbf{B}:=\sum_{p\le z}\mathbf{B}_p=\mathbf{S}^{(\omega,\omega)}.
\]

\begin{proposition}
\label{comparison-models}
For sufficiently large $z$, and 
$(\nu_0,\nu_1)\in\{\omega,\Omega\}^2$,
\begin{equation}
\sup_{\mathbf{a}\in\mathbb Z^2}
\left|
\mathbb P\!\left(\mathbf{S}^{(\nu_0,\nu_1)}=\mathbf{a}\right)
-
\mathbb P(\mathbf{B}=\mathbf{a})
\right|
\ll
(\log_2 z)^{-3/2}.
\end{equation}
\end{proposition}

\begin{proof}
Fix $(\nu_0,\nu_1)\in\{\omega,\Omega\}^2$. Note that the case $(\nu_0,\nu_1)=(\omega,\omega)$ is trivial. We nevertheless include it in the argument below in order to keep the notation uniform across all four models. For
$\mathbf{t}=(t_0,t_1)\in[-1/2,1/2]^2$, set the local characteristic functions
\begin{align*}
\phi_{p,\nu}(\mathbf{t})
:=
\mathbf{E} e\!\left(\mathbf{t}\cdot \mathbf{Y}_p^{(\nu_0,\nu_1)}\right),
\qquad
\phi_{p,\mathbf{B}}(\mathbf{t})
:=
\mathbf{E} e(\mathbf{t}\cdot \mathbf{B}_p),
\end{align*}
and let
\begin{align*}
\Phi_{z,\nu}(\mathbf{t})
:=
\prod_{p\le z}\phi_{p,\nu}(\mathbf{t}),
\qquad
\Phi_{z,\mathbf{B}}(\mathbf{t})
:=
\prod_{p\le z}\phi_{p,\mathbf{B}}(\mathbf{t})
\end{align*}
be the corresponding full characteristic functions. Under the coupling introduced above, define
\begin{align}
    \mathbf{D}_p^{(\nu_0,\nu_1)}
:=
\mathbf{Y}_p^{(\nu_0,\nu_1)}-\mathbf{B}_p.
\end{align}
The components of the random vector are
\begin{align}
\left(\mathbf{D}_p^{(\nu_0,\nu_1)}\right)_i
=
\mathbf 1_{\{\nu_i=\Omega\}}(A_{i,p}-1)_+,
\qquad i=0,1.
\end{align}
Since $(A_{i,p}-1)_+ = A_{i,p}-\mathbf 1_{\{A_{i,p}\ge1\}},$
the calculations in Section \ref{subsubsec:llt-moments} give
\[
\mathbf{E}(A_{i,p}-1)_+
=
\frac1{p-1}-\frac1p
=
\frac1{p(p-1)}.
\]
Hence
\[
\mathbf{E}\left\|\mathbf{D}_p^{(\nu_0,\nu_1)}\right\|
\ll \frac1{p(p-1)}.
\]
Using the elementary inequality $|e(u)-e(v)|\le 2\pi|u-v|$ for $u, v \in \R$, and the Cauchy-Schwarz inequality, we obtain
\begin{align}
\left|\phi_{p,\nu}(\mathbf{t})-\phi_{p,\mathbf{B}}(\mathbf{t})\right|
&\leq
\mathbf{E}\left|
e\!\left(\mathbf{t}\cdot \mathbf{Y}_p^{(\nu_0,\nu_1)}\right)
-e(\mathbf{t}\cdot \mathbf{B}_p)
\right| \notag\\
&\leq
2\pi\,\mathbf{E}\left|\mathbf{t}\cdot \mathbf{D}_p^{(\nu_0,\nu_1)}\right| \notag\\
&\leq
2\pi\|\mathbf{t}\|\,\mathbf{E}\left\|\mathbf{D}_p^{(\nu_0,\nu_1)}\right\|
\ll
\frac{\|\mathbf{t}\|}{p(p-1)}.
\label{eq:char-1}
\end{align}
We next record a uniform decay estimate for the local characteristic functions. For any of the four pairs and $p \geq 3$,
\begin{align*}
\mathbf{P}(\mathbf{Y}_p=\mathbf{0})=1-\frac2p,
\qquad
\mathbf{P}(\mathbf{Y}_p=\mathbf{e}_i)
\ge
\left(1-\frac1p\right)\frac1p,
\qquad i=0,1,
\end{align*}
where $\mathbf{e}_0=(1,0)$ and $\mathbf{e}_1=(0,1)$. If $\mathbf{Y}_p'$ is an independent copy of
$\mathbf{Y}_p$, then
\begin{align*}
1 - \left|\mathbf{E} e(\mathbf{t}\cdot \mathbf{Y}_p)\right|^2
=
1 - \mathbf{E} e(\mathbf{t}\cdot \mathbf{Y}_p) \overline{\mathbf{E} e(\mathbf{t}\cdot \mathbf{Y}_p)}
&=
1 - \mathbf{E}
\left[
\cos\!\left(2\pi \mathbf{t}\cdot(\mathbf{Y}_p-\mathbf{Y}_p')\right)
\right] \\
&\gg
\frac{1-\cos(2\pi t_0)}p
+
\frac{1-\cos(2\pi t_1)}p \\
&\gg
\frac{\|\mathbf{t}\|^2}{p},
\end{align*}
where the last inequality follows from the basic inequality $1 - \cos{2\pi u} \gg u^2$ which holds whenever $\left|u\right| \leq 1/2$. Therefore, for some absolute $c>0$,
\begin{equation}
\left|\mathbf{E} e(\mathbf{t}\cdot \mathbf{Y}_p)\right|
\le
\exp\!\left(-\frac{c\|\mathbf{t}\|^2}{p}\right),
\qquad p\geq 3.
\label{eq:char-2}
\end{equation}
We now pass from the local comparison to the characteristic functions $\Phi_{z,\nu}$ and $\Phi_{z,\mathbf B}$.  Ordering the primes increasingly and telescoping the two products gives
\begin{align}
\Phi_{z,\nu}(\mathbf{t})-\Phi_{z,\mathbf{B}}(\mathbf{t})
=
\sum_{q\le z}
\bigl(\phi_{q,\nu}(\mathbf{t})-\phi_{q,\mathbf{B}}(\mathbf{t})\bigr)
\prod_{p<q}\phi_{p,\nu}(\mathbf{t})
\prod_{q<p\le z}\phi_{p,\mathbf{B}}(\mathbf{t}).
\label{eq:characteristic-telescoping}
\end{align}
Now, fix a prime $\ell \le z$, and consider the summand corresponding to $q = \ell$ in \eqref{eq:characteristic-telescoping}. Combining the two estimates \eqref{eq:char-1} and \eqref{eq:char-2}, the summand corresponding to $q=\ell$ is bounded by
\begin{align}
\ll
\frac{\|\mathbf t\|}{\ell(\ell-1)}
e^{-c_1(\log_2 z)\|\mathbf t\|^2},
\label{eq:characteristic-telescoping-bound}
\end{align}
for some $c_1 > 0$. Taking absolute values in the telescoping identity and using \eqref{eq:characteristic-telescoping-bound}, we obtain
\begin{align}
|\Phi_{z,\nu}(\mathbf{t})-\Phi_{z,\mathbf{B}}(\mathbf{t})|
&\ll
\sum_{q\le z}
\frac{\|\mathbf{t}\|}{q(q-1)}
\exp\!\left(
-c_1(\log_2 z)\|\mathbf{t}\|^2
\right) \notag\\
&\ll
\|\mathbf{t}\|
\exp\!\left(
-c_1(\log_2 z)\|\mathbf{t}\|^2
\right).
\end{align}
Finally, Fourier inversion on $\Z^2$ gives, uniformly in $\mathbf{a}\in\mathbb Z^2$,
\begin{align*}
\left|
\mathbf{P}\!\left(\mathbf{S}^{(\nu_0,\nu_1)}=\mathbf{a}\right)
-
\mathbf{P}(\mathbf{B}=\mathbf{a})
\right|
&\le
\int_{[-1/2,1/2]^2}
\left|
\Phi_{z,\nu}(\mathbf{t})-\Phi_{z,\mathbf{B}}(\mathbf{t})
\right|\,d\mathbf{t} \\
&\ll
\int_{\mathbb R^2}
\|\mathbf{t}\|
e^{-c_1(\log_2 z)\|\mathbf{t}\|^2}\,d\mathbf{t} \\
&\ll
(\log_2 z)^{-3/2},
\end{align*}
where in the last step we make the change of variables
$\mathbf{u}=\sqrt{\log_2 z}\,\mathbf{t}$.
\end{proof}
We record for later use that the same argument gives
\begin{align}
|\Phi_{z,\nu}(\mathbf{t})|
\ll
\exp\!\left(-c(\log_2 z)\|\mathbf{t}\|^2\right),
\qquad
\mathbf{t}\in[-1/2,1/2]^2,
\label{eq:global-decay-estimate}
\end{align}
uniformly over all four pairs.


\subsubsection{The \texorpdfstring{$(\omega,\omega)$}{omega--omega} Local Limit Theorem}
\label{subsubsec:llt-bernoulli-proof}

Proposition \ref{comparison-models} tells us that if one shows
\[
\mathbf{P}(\mathbf{B}=\mathbf{a})
=
G_{\log_2 z}(\mathbf{a})
+
O\!\left((\log_2 z)^{-3/2}\right)
\]
uniformly in $\mathbf{a}\in\mathbb Z^2$, then the same estimate holds for all four
models $\mathbf{S}^{(\nu_0,\nu_1)}$. Therefore, it remains to establish a local limit theorem for $\mathbf{B}$. 
To that end, set $\mathbf{Z}_p:=\mathbf{B}_p-\mathbf{E} \mathbf{B}_p$ and define the centered characteristic function for $\boldsymbol{t} \in \R^2$
\begin{align}
 \Psi_z(\boldsymbol t)
 :=\mathbf{E} e\bigl(\boldsymbol t\cdot(\mathbf{B}-\boldsymbol\mu_\mathbf{B})\bigr)
 =\prod_{p\leq z}\mathbf{E} e(\boldsymbol t\cdot \mathbf{Z}_p).
\end{align}
Note that $|\Phi_{z,\mathbf{B}}(\mathbf{t})| = |\Psi_z(\mathbf{t})|$.

Before stating and proving the $(\omega, \omega)$ local limit theorem, we record one additional technical result concerning the second-order expansion of the logarithm of a characteristic function, based on \cite[Lemma 4.10]{kallenberg2002}. We explicitly include the uniform remainder and the choice of logarithmic branch needed below.

\begin{lemma}
\label{lem:complex-log}
There exists an absolute constant $\rho>0$ such that, uniformly for
$\|\mathbf t\|\leq \rho$,
\[
\log \Psi_z(\mathbf{t})
=
-2\pi^2 \mathbf{t}^{\mathsf T}\Sigma_{\mathbf{B}} \mathbf{t}
+
O\!\left(\log_2z\|\mathbf{t}\|^3\right).
\]
Here the logarithm denotes the continuous branch equal to zero at the
origin, obtained by summing the principal logarithms of the local centered
characteristic functions.
\end{lemma}

\begin{proof}
For each prime $p\leq z$, write
\[
\psi_p(\mathbf{t}):=\mathbf{E} e(\mathbf{t}\cdot \mathbf{Z}_p),
\qquad
u_p(\mathbf{t}):=\psi_p(\mathbf{t})-1.
\]
Since $\mathbf{E} \mathbf{Z}_p=0$, the second-order Taylor expansion
\[
e(v)
=
1+2\pi i v-2\pi^2v^2+O(|v|^3),
\qquad v\in\mathbb R,
\]
applied with $v=\mathbf{t}\cdot \mathbf{Z}_p$, gives
\[
\psi_p(\mathbf{t})
=
1
-
2\pi^2\mathbf{E}(\mathbf{t}\cdot \mathbf{Z}_p)^2
+
O\!\left(\mathbf{E}|\mathbf{t}\cdot \mathbf{Z}_p|^3
\right).
\]
Now, $\mathbf{E}(\mathbf{t}\cdot \mathbf{Z}_p)^2 = \mathbf{t}^{\mathsf T}\operatorname{\mathbf{Cov}}(\mathbf{B}_p)\mathbf{t}$, while $|\mathbf{t}\cdot \mathbf{Z}_p|
\leq \|\mathbf{t}\|\,\|\mathbf{Z}_p\|$. Hence,
\begin{equation}
u_p(\mathbf{t})
=
-2\pi^2
\mathbf{t}^{\mathsf T}\operatorname{\mathbf{Cov}}(\mathbf{B}_p)\mathbf{t}
+
O\!\left(
\|\mathbf{t}\|^3\mathbb E\|\mathbf{Z}_p\|^3
\right).
\label{eq:Taylor-Characteristic}
\end{equation}
We also need a uniform bound showing that $\psi_p$ remain close
to $1$. Since
\begin{align*}
|e(v)-1-2\pi i v|
\leq 2\pi^2v^2,
\qquad v\in\mathbb R,
\end{align*}
and $\mathbf{E} \mathbf{Z}_p=0$, we obtain
\[
\begin{aligned}
|u_p(\mathbf{t})|
&=
\left|
\mathbf{E}
\bigl(
e(\mathbf{t}\cdot \mathbf{Z}_p)-1-2\pi i\,\mathbf{t}\cdot \mathbf{Z}_p
\bigr)
\right| \\
&\leq
2\pi^2\mathbf{E}|\mathbf{t}\cdot \mathbf{Z}_p|^2 \\
&\leq
2\pi^2\|\mathbf{t}\|^2\mathbf{E}\|\mathbf{Z}_p\|^2
\ll
\frac{\|\mathbf{t}\|^2}{p},
\end{aligned}
\label{eq:Taylor-Bound}
\]
where we used the centered second-moment bound from \eqref{eq:centered-moments}. Choose $\rho\leq1$ sufficiently small so that $|u_p(\mathbf{t})|\leq 1/2$ for every prime $p$ whenever $\|\mathbf{t}\|\leq\rho$. Then, each local factor $1+u_p(\mathbf{t})$ lies in the disk
\[
\{w\in\mathbb C:|w-1|\leq1/2\},
\]
on which the principal logarithm is analytic and satisfies $\operatorname{Log}(1+u)=u+O(|u|^2)$ \cite[Chapter 3.6]{stein2003complex}.
We may therefore define
\begin{align}
\mathcal L_z(\mathbf{t})
:=
\sum_{p\leq z}\operatorname{Log}(1+u_p(\mathbf{t})).
\end{align}
$\mathcal L_z$ is continuous for $\|\mathbf{t}\|\leq\rho$, satisfies
$\mathcal L_z(\mathbf 0)=0$, and
\[
\exp(\mathcal L_z(\mathbf{t}))
=
\prod_{p\leq z}(1+u_p(\mathbf{t}))
=
\prod_{p\leq z}\psi_p(\mathbf{t})
=
\Psi_z(\mathbf{t}).
\]
Thus $\mathcal L_z$ is precisely the branch of $\log\Psi_z$ specified in the statement. Using \eqref{eq:Taylor-Characteristic} and $\operatorname{Log}(1+u_p)=u_p+O(|u_p|^2)$, we obtain
\begin{align}
\log\Psi_z(\mathbf{t})
&=
-2\pi^2
\mathbf{t}^{\mathsf T}
\left(
\sum_{p\leq z}\operatorname{\mathbf{Cov}}(\mathbf{B}_p)
\right)
\mathbf{t} \notag \\
&\quad
+
O\!\left(
\|\mathbf{t}\|^3
\sum_{p\leq z}\mathbf{E}\|\mathbf{Z}_p\|^3
+
\sum_{p\leq z}|u_p(\mathbf{t})|^2
\right).
\end{align}
Since the $B_p$ are independent, $\sum_{p\leq z}\boldsymbol{\operatorname{Cov}}(\mathbf{B}_p) = \Sigma_\mathbf{B}$. Moreover, $|u_p(\mathbf t)|\ll \|\mathbf t\|^2/p$ gives
\begin{align*}
\sum_{p\leq z}|u_p(\mathbf{t})|^2
\ll
\|\mathbf{t}\|^4
\sum_p\frac1{p^2}
\ll
\|\mathbf{t}\|^4.
\end{align*}
Therefore, together with the moment estimates of \eqref{eq:centered-moments}, we obtain
\begin{align*}
\log\Psi_z(\mathbf{t})
=
-2\pi^2\mathbf{t}^{\mathsf T}\Sigma_B\mathbf{t}
+
O\!\left(
\log_2z\|\mathbf{t}\|^3+\|\mathbf{t}\|^4
\right).
\end{align*}
Since $\rho\leq1$ and $\log_2z\geq1$ for sufficiently large $z$, $\|\mathbf{t}\|^4 \leq
\log_2z\|\mathbf{t}\|^3$, and hence
\begin{align*}
\log\Psi_z(\mathbf{t})
=
-2\pi^2\mathbf{t}^{\mathsf T}\Sigma_\mathbf{B}\mathbf{t}
+
O\!\left(
\log_2z\|\mathbf{t}\|^3
\right).
\end{align*}
\end{proof}

\begin{theorem}[Bivariate Local Limit Theorem]
\label{thm:2d-llt}
For $x\in I_X$ and $\mathbf{a}\in\mathbb Z^2$,
\[
\mathbf{P}\!\left(\mathbf{B}=\mathbf{a}\right)
=
G_{\log_2 z}(\mathbf{a})
+
O\!\left((\log_2 z)^{-3/2}\right).
\]
\end{theorem}

\begin{proof}

By Fourier inversion on $\Z^2$,
\begin{align}
\mathbf{P}(\mathbf{B}=\mathbf{a})
=
\int_{[-1/2,1/2]^2}
e\!\left(-\mathbf{t}\cdot(\mathbf{a}-\mu_{\mathbf{B}})\right)\Psi_z(t)\,d\mathbf{t}.
\end{align}
We make the change of variables $\mathbf{t} = \mathbf{u} / \sqrt{\log_2 z}$ and set
\begin{align*}
Q_z
:=
\left[
-\frac{\sqrt{\log_2 z}}2,
\frac{\sqrt{\log_2 z}}2
\right]^2.
\end{align*}
It follows that
\begin{align}
\mathbf{P}(\mathbf{B}=\mathbf{a})
=
\frac1{\log_2 z}
\int_{Q_z}
e\!\left(
-\mathbf{u}\cdot
\bigg(\frac{\mathbf{a}-\mu_\mathbf{B}}{\sqrt{\log_2 z}}\bigg)
\right)
\Psi_z\!\left(
\frac{\mathbf{u}}{\sqrt{\log_2 z}}
\right)
\,d\mathbf{u}.
\end{align}
Now, set the scale $U:=(\log_2 z)^{1/12}$, so that $U/ \sqrt{\log_2z} \rightarrow 0$.  Thus, for $\|\mathbf{u}\|\le U$, Lemma \ref{lem:complex-log} applies for all sufficiently large $z$. Since $\Sigma_{\mathbf{B}}=(\log_2 z)I_2+O(1)$, we obtain
\begin{align}
\log
\Psi_z\!\left(
\frac{\mathbf{u}}{\sqrt{\log_2 z}}
\right)
=
-2\pi^2\|\mathbf{u}\|^2+E_z(\mathbf{u}),
\label{eq:complex_psi_z}
\end{align}
where
\begin{align*}
|E_z(\mathbf{u})|
\ll
\frac{\|\mathbf{u}\|^2}{\log_2 z}
+
\frac{\|\mathbf{u}\|^3}{\sqrt{\log_2 z}}.
\end{align*}
In fact, since $\|\mathbf{u}\|\le U$, $|E_z(\mathbf{u})| / \|\mathbf{u}\|^2 \ll 1 / \log_2z + U / \sqrt{\log_2z} = o(1)$. Hence, for sufficiently large $z$ and every $s\in[0,1]$,
\begin{align}
    \Re\!\left(
-2\pi^2\|\mathbf{u}\|^2+sE_z(u)
\right)
\le
-\pi^2\|\mathbf{u}\|^2.
\label{eq:gaussian-error-realpart}
\end{align}
We now aim to compare \eqref{eq:complex_psi_z} with the Gaussian $
e^{-2\pi^2\|\mathbf{u}\|^2}$.  Using \cite[Theorem 1.3.2]{stein2003complex}, we have
\begin{align*}
    \exp(v+w)-\exp(v) = w\int_0^1\exp(v+sw)\,ds,
\end{align*}
for $v, w \in \C$, and together with \eqref{eq:gaussian-error-realpart}, we obtain
\begin{align*}
\left|
\Psi_z\!\left(
\frac{\mathbf{u}}{\sqrt{\log_2 z}}
\right)
-
e^{-2\pi^2\|\mathbf{u}\|^2}
\right|
\ll
e^{-\pi^2\|\mathbf{u}\|^2}
\left(
\frac{\|\mathbf{u}\|^2}{\log_2 z}
+
\frac{\|\mathbf{u}\|^3}{\sqrt{\log_2 z}}
\right).
\end{align*}
Since Gaussian moments are finite, i.e. $\int_{\mathbb R^2} \|\mathbf{u}\|^k e^{-\pi^2\|\mathbf{u}\|^2}\,d\mathbf{u}<\infty$, integrating over $\|\mathbf{u}\|\leq U$ yields
\begin{align*}
\int_{\|\mathbf{u}\|\le U}
\left|
\Psi_z\!\left(
\frac{\mathbf{u}}{\sqrt{\log_2 z}}
\right)
-
e^{-2\pi^2\|\mathbf{u}\|^2}
\right|
\,d\mathbf u
\ll
(\log_2 z)^{-1/2}.
\end{align*}
It remains to treat the case $\|\mathbf{u}\|>U$. For this, the decay estimate obtained in \eqref{eq:global-decay-estimate} gives an absolute constant $c > 0$ such that
\[
\left|
\Psi_z\!\left(
\frac{\mathbf{u}}{\sqrt{\log_2 z}}
\right)
\right|
\ll
e^{-c\|\mathbf{u}\|^2},
\qquad \mathbf{u}\in Q_z.
\]
Therefore,
\[
\int_{Q_z\cap\{\|\mathbf{u}\|>U\}}
\left(
\left|
\Psi_z\!\left(
\frac{\mathbf{u}}{\sqrt{\log_2 z}}
\right)
\right|
+
e^{-2\pi^2\|\mathbf{u}\|^2}
\right)
\,du
\ll
e^{-c'U^2}
\]
for some constant $c' > 0$, which shows that the contribution of the region $\|\mathbf{u}\|>U$ is negligible. 
Extending the Gaussian integral from $Q_z$ to $\mathbb R^2$ produces only
an exponentially small error. Hence, for $\mathbf{a}\in\mathbb Z^2$,
\[
\begin{aligned}
\mathbf{P}(\mathbf{B}=\mathbf{a})
&=
\frac1{\log_2 z}
\int_{\mathbb R^2}
e\!\left(
-\mathbf{u}\cdot
\frac{\mathbf{a}-\mu_{\mathbf{B}}}{\sqrt{\log_2 z}}
\right)
e^{-2\pi^2\|\mathbf{u}\|^2}
\,d\mathbf{u}
+
O\!\left((\log_2 z)^{-3/2}\right).
\end{aligned}
\]
Using the two-dimensional Gaussian identity gives
\begin{align}
\mathbf{P}(\mathbf{B}=\mathbf{a})
=
\frac1{2\pi\log_2 z}
\exp\!\left(
-\frac{\|\mathbf{a}-\mu_{\mathbf{B}}\|^2}{2\log_2 z}
\right)
+
O\!\left((\log_2 z)^{-3/2}\right).
\end{align}
By the moment calculations in Section \ref{subsubsec:llt-moments}, $\mu_\mathbf{B}=(\log_2 z,\log_2 z)+O(1)$. Together with the uniform bound $\|\nabla G_t\|_\infty\ll t^{-3/2}$ from the preliminaries, we may replace $\mu_\mathbf{B}$ by
$(\log_2 z,\log_2 z)$ at a cost
$O((\log_2 z)^{-3/2})$. Thus
\[
\mathbf P(\mathbf B=\mathbf{a})
=
G_{\log_2 z}(\mathbf{a})
+
O\!\left((\log_2 z)^{-3/2}\right).
\]
\end{proof}

\begin{remark}
Utilising the estimate $\|\partial_t G_t\|_\infty\ll t^{-3/2}$
given in Section \ref{subsec:prelim-Fourier}, the mean value theorem gives
\begin{align}
\sup_{\mathbf{a}\in\mathbb R^2}
\left|G_{\log_2 z}(\mathbf{a})-G_{\log_2 x}(\mathbf{a})\right|
\ll
(\log R)(\log_2 x)^{-3/2}.
\end{align}
Combining this with the preceding theorem  and using $\log_2 z\asymp\log_2 x$, we obtain
\begin{align}
\mathbf{P}\!\left(\mathbf{S}^{(\nu_0,\nu_1)}=\mathbf{a}\right)
=
G_{\log_2 x}(\mathbf{a})
+
O\!\left((\log R)(\log_2 x)^{-3/2}\right).
\label{eq:preconcentration-estimate}
\end{align}
The theorem also yields the following uniform concentration estimate
\begin{align}
\max_{(\nu_0,\nu_1)\in\{\omega,\Omega\}^2}
\sup_{a\in\mathbb Z^2} \mathbf{P}
\mathbf\!\left(\mathbf{S}^{(\nu_0,\nu_1)}= \mathbf a\right)
\ll
(\log_2 x)^{-1}.
\label{eq:concentration-estimate}
\end{align}
\end{remark}

\subsection{Completion of the Low-frequency Estimate}

We now return to the truncation at the scale $z$ that we postponed in Section \ref{subsubsec:low-truncation}. The concentration estimate \eqref{eq:concentration-estimate} supplied by Theorem \ref{thm:2d-llt} allows us to recover the saving that was missing from the crude estimate in Section \ref{subsubsec:low-truncation}. We then evaluate the resulting truncated expression by Gaussian smoothing.

\begin{proposition}[Truncation] \label{prop:truncation-main}
For $x\in I_X$, $(\nu_0,\nu_1)\in\{\omega,\Omega\}^2$, and
$\mathbf m =(m_0,m_1)\in\mathbb Z^2$, it holds that
\[
\begin{aligned}
\frac1x\sum_{x<n\leq 2x}
\Big|
&T_R\!\left(\nu_0(n)-m_0\right)
 T_R\!\left(\nu_1(n+1)-m_1\right) \\
&-
T_R\!\left(\nu_0(n_{\leq z})-m_0\right)
T_R\!\left(\nu_1((n+1)_{\leq z})-m_1\right)
\Big|
\ll
\frac{1}{(\log_2 x)\sqrt R}.
\end{aligned}
\]
Consequently,
\[
L_x^{\nu_0,\nu_1}(\mathbf m)
=
\widetilde L_x^{\nu_0,\nu_1}(\mathbf m)
+
O\!\left(\frac{1}{(\log_2x)\sqrt R}\right).
\]
\end{proposition}

\begin{proof}
Recall from the pointwise estimate preceding \eqref{eq:naive-truncation}, we obtained
\begin{align}
&\frac1x\sum_{x<n\leq2x}
\Big|
T_R(\nu_0(n)-m_0)T_R(\nu_1(n+1)-m_1) \notag\\
&\qquad
-
T_R(\nu_0(n_{\leq z})-m_0)
T_R(\nu_1((n+1)_{\leq z})-m_1)
\Big| \notag\\
&\qquad\ll
R^{-5/2}\frac1x\sum_{x<n\leq2x}
W_R(u_0(n))W_R(u_1(n)).
\label{eq:truncation-target-final}
\end{align}
where $u_i(n):=\nu_i((n+i)_{\le z})-m_i$ for $i \in \{0,1\}$, and
\begin{align*}
    W_R(u):=
\left(1+\frac{|u|}{R}\right)^{-A},
\end{align*}
with $A$ being a sufficiently large integer. It therefore remains to estimate the right-hand side of \eqref{eq:truncation-target-final}. This is where Proposition \ref{prop:low-sieve-transfer} being stated for general bounded $F: \Z^2 \mapsto \C$ and the concentration estimate \eqref{eq:concentration-estimate} come directly in use. Namely, for $F(a_0,a_1) :=
W_R(a_0-m_0)W_R(a_1-m_1)$,
\[
\begin{aligned}
\frac1x\sum_{x<n\leq2x}
W_R(u_0(n))W_R(u_1(n))
&=
\mathbf E\!\left[
W_R(S_0-m_0)W_R(S_1-m_1)
\right] +
O\!\left(e^{-cR^{1/4}}\right),
\end{aligned}
\]
where, for the rest of the proof, $\mathbf S=\mathbf{S}^{(\nu_0,\nu_1)} = (S_0, S_1)$. On the other hand, \eqref{eq:concentration-estimate} allows us to obtain 
\begin{align}
\mathbf E\!\left[
W_R(S_0-m_0)W_R(S_1-m_1)
\right]
\ll
\frac{R^2}{\log_2x}.
\label{eq:truncation-estimate-prefinal}
\end{align}
Combining these estimates, and substituting back into \eqref{eq:truncation-target-final}, we obtain
\[
\ll
R^{-5/2}
\left(
\frac{R^2}{\log_2x}
+
e^{-cR^{1/4}}
\right)
\ll
\frac{1}{(\log_2x)\sqrt R}.
\]
Here the exponentially small term is negligible since
$\log_2x\asymp\log_2X=R^4$.
The final assertion follows immediately from the triangle inequality.
\end{proof}

\begin{lemma}[Gaussian smoothing]
\label{lem:Gaussian-Smoothing}
For every $y\geq2$, uniformly in $\mathbf m\in\mathbb Z^2$,
\begin{align}
\sum_{\mathbf a\in\mathbb Z^2}
T_R(a_0-m_0)T_R(a_1-m_1)G_y(\mathbf a)
=
G_y(\mathbf m)
+
O(R^2y^{-2}).
\label{eq:Gaussian-Smoothing}
\end{align}
\end{lemma}

\begin{proof}
Writing $\mathbf h=\mathbf a- \mathbf m$, the left-hand side of \eqref{eq:Gaussian-Smoothing} becomes
\begin{align}
\sum_{h\in\mathbb Z^2}
T_R(h_0)T_R(h_1)G_y(\mathbf m+ \mathbf h).
\label{eq:Rewritten-Smoothing}
\end{align}
By the Hessian bound $\|D^2G_y\|_\infty\ll y^{-2}$ from the preliminaries, Taylor's theorem gives
\begin{align}
G_y(\mathbf m+ \mathbf h)
=
G_y(\mathbf m)
+
\nabla G_y(\mathbf m)\cdot \mathbf h
+
O(y^{-2}\|\mathbf h\|^2)\label{eq:Taylor-Smoothing}
\end{align}
uniformly in $\mathbf m$ and $\mathbf h$.

By Lemma \ref{lem:smooth-lattice-kernel}, $\sum_{j\in\mathbb Z}T_R(j)=1, \,
\sum_{j\in\mathbb Z}jT_R(j)=0$. Hence,
\begin{align*}
\sum_{h\in\mathbb Z^2}
T_R(h_0)T_R(h_1)=1, \qquad
\sum_{h\in\mathbb Z^2}
h_iT_R(h_0)T_R(h_1)=0 \quad \text{for}\,\, i=0,1.
\end{align*}
Therefore, inserting \eqref{eq:Taylor-Smoothing} into \eqref{eq:Rewritten-Smoothing} gives that the constant term contributes $G_y(m)$ only, the linear term vanishes, and for the quadratic remainder, using Lemma \ref{lem:smooth-lattice-kernel},
\begin{align*}
\sum_{h\in\mathbb Z^2}
\|h\|^2
|T_R(h_0)T_R(h_1)|
&=
2
\left(
\sum_{j\in\mathbb Z}j^2|T_R(j)|
\right)
\left(
\sum_{j\in\mathbb Z}|T_R(j)|
\right) \\
&\ll
R^2.
\end{align*}
Therefore the total remainder is $O(R^2y^{-2})$, as claimed.
\end{proof}
We now combine the preceding results to complete the low-frequency estimate. By Proposition \ref{prop:truncation-main},
\[
L_x^{\nu_0,\nu_1}(\mathbf m)
=
\widetilde L_x^{\nu_0,\nu_1}(\mathbf m)
+
O\left(\frac{1}{(\log_2 x)\sqrt R}\right).
\]
On the other hand, applying Proposition \ref{prop:low-sieve-transfer} with $F(a_0,a_1)
:= T_R(a_0-m_0)T_R(a_1-m_1)$,
for which \(\|F\|_\infty\ll R^{-2}\), together with the estimate \eqref{eq:preconcentration-estimate} and Lemma \ref{lem:smooth-lattice-kernel} gives
\begin{align}
\widetilde L_x^{\nu_0,\nu_1}(\mathbf m)
&=
\sum_{a\in\mathbb Z^2}
T_R(a_0-m_0)T_R(a_1-m_1)G_{\log_2 x}(a) \\
& \notag\quad+
O\!\left(
(\log R)(\log_2 x)^{-3/2}
+
R^{-2}e^{-cR^{1/4}}
\right).
\end{align}
Finally, using Lemma \ref{lem:Gaussian-Smoothing} with $y=\log_2 x$, we conclude that
\begin{align}
L_x^{\nu_0,\nu_1}(\mathbf m)
&=
G_{\log_2 x}(\mathbf m) \notag\\
& +
O\left(
\frac{1}{(\log_2 x)\sqrt R}
+
(\log R)(\log_2 x)^{-3/2}
+
R^2(\log_2 x)^{-2}
+
R^{-2}e^{-cR^{1/4}}
\right).
\end{align}
We recall one more time that $R = (\log_2X)^{1/4}$ and $ \log_2 x \asymp \log_2 X$ for \(x\in I_X\). Therefore, handling the error term yields 
\[
L_x^{\nu_0,\nu_1}(\mathbf m)
=
G_{\log_2 x}(\mathbf m)
+
O\!\left((\log_2 X)^{-9/8}\right), 
\]
which implies the required low-frequency estimate with \(c_2=1/8\).

\section{Dynamical Results}
\label{sec:dynamical}

In this section we prove Theorem \ref{thm:joint-dynamical-limit}. To that end, we recall the statement of the theorem. Namely, we are given the set $\mathcal{X}$ obtained from Theorem \ref{thm:intro-main}. We want to show that for a uniquely ergodic $\N^2$-system, denoted by $(Y, \mu, T_0, T_1$), for every $F \in C(Y)$
\begin{align}
\label{eq:main-joint-dynamics}
\max_{\nu_0,\nu_1\in\{\omega,\Omega\}}
\sup_{y\in Y}
\left|
\frac1N\sum_{n\le N}
F\bigl(T_0^{\nu_0(n)}T_1^{\nu_1(n+1)}y\bigr)
-\int_Y F\,d\mu
\right|\longrightarrow 0,
\end{align}
as $N \to \infty$ with $N \in \mathcal{X}$. Additionally, we aim to show 
\begin{align}
\label{eq:main-joint-log-dynamics}
\max_{\nu_0,\nu_1\in\{\omega,\Omega\}}
\sup_{y\in Y}
\left|
\frac1{\log N}\sum_{n\le N}
\frac{F\bigl(T_0^{\nu_0(n)}T_1^{\nu_1(n+1)}y\bigr)}{n}
-\int_Y F\,d\mu
\right|\longrightarrow 0
\end{align}
as $N\to\infty$.

Our proof follows the strategy described in \cite[Remark 1.3]{BergelsonRichter2022}. We first rewrite the average in \eqref{eq:main-joint-dynamics} as weighted orbit average. Namely, notice that
\begin{align}
    \frac1N\sum_{n\le N}
F\bigl(T_0^{\nu_0(n)}T_1^{\nu_1(n+1)}y\bigr) = \sum_{\mathbf{m} \in \N_0^2} P_N^{\nu_0, \nu_1}(\mathbf{m})F(T_0^{m_0}T_1^{m_1}y)
\end{align}
where the weights are precisely the normalized joint counts of $\nu_0(n)$ and $\nu_1(n+1)$ for $\nu_0,\nu_1\in\{\omega,\Omega\}$. Our local limit theorem allows us to
show that these weights become asymptotically invariant, in
$\ell^1(\mathbb Z^2)$, under a unit shift in either coordinate.
Consequently, every weak-$*$ limit of the associated orbit measures is invariant under both coordinate transformations. Joint unique ergodicity allows us to conclude our result.  Finally, partial summation and the logarithmic
density-one property of $\mathcal X$ yield the logarithmically averaged conclusion in \eqref{eq:main-joint-log-dynamics}.

The necessary background needed for the proof is recorded in Section~\ref{subsec:dyn-preliminaries}, and the proof is given in
Section~\ref{subsec:dyn-proof}.

\subsection{Preliminary Results}\label{subsec:dyn-preliminaries}

We first recall some standard notions and results concerning
uniquely ergodic topological dynamical systems, weak-$*$ convergence, pushforward measures and total variation. For a more detailed background, we refer to \cite{EinsiedlerWard2011, Walters1982}.

Let $Y$ be a compact metric space, and let $T_0,T_1:Y\to Y$ be
two commuting continuous maps. For $\mathbf m=(m_0,m_1)\in\mathbb N^2$, define $T^{\mathbf m}:=T_0^{m_0}\circ T_1^{m_1}$. By commutativity, $T^{\mathbf m+\mathbf k}=T^{\mathbf m}\circ T^{\mathbf k}$, so these maps define an action of the additive
semigroup $\mathbb N^2$ on $Y$. We denote by $\mathcal P(Y)$ the
set of Borel probability measures on $Y$. A Borel probability measure $\mu$ on $Y$ is invariant under this action if it is invariant under both $T_0$ and $T_1$, that is, if $\int_Y F\circ T_i\,d\mu=\int_Y F\,d\mu$, with $i \in \{0, 1\}$, for every $F\in C(Y)$. A straightforward adaptation of the Krylov–Bogolyubov argument to the action of $\N^2$ shows that there exists at least one such Borel probability measure on $Y$ (cf.  \cite[Theorem 4.1 and Corollary 4.2]{EinsiedlerWard2011}). We call the action uniquely ergodic if there is exactly one Borel probability measure invariant under both transformations. In this case, we
write $(Y,\mu,T_0,T_1)$ for the system.


We collect some facts regarding weak$*$ convergence and pushforward measures. For $y\in Y$, we denote by $\delta_y$ the Dirac probability measure, which is characterized
by $\int_Y F\,d\delta_y=F(y)$ with $F \in C(Y)$. We say that a sequence $(\eta_j)\subset\mathcal P(Y)$ converges weak-$*$ to
$\eta\in\mathcal P(Y)$, written
$\eta_j\stackrel{*}{\rightharpoonup}\eta$, if
\begin{equation}\label{eq:dyn-weakstar-definition}
 \int_Y F\,d\eta_j\longrightarrow\int_Y F\,d\eta
 \qquad\text{for every }F\in C(Y).
\end{equation}
We record the following standard result \cite[Theorem~B.11(2)--(4)]{EinsiedlerWard2011}. 
\begin{proposition}
\leavevmode
\label{prop:dyn-measure-compactness}
\begin{enumerate}
    \item Continuous functions distinguish Borel probability measures on $Y$, i.e., for any $\mu_1, \mu_2 \in \mathcal P(Y)$,
    \begin{align}
        \mu_1=\mu_2
\quad\Longleftrightarrow\quad
\int_Y F\,d\mu_1=\int_Y F\,d\mu_2
\qquad\text{for every }F\in C(Y).
    \end{align}
    \item The weak-$*$ topology on $\mathcal P(Y)$ is compact and metrizable. By sequential compactness, every sequence in $\mathcal P(Y)$ has a weak-$*$ convergent subsequence whose limit is again a probability measure.
\end{enumerate}
\end{proposition}
For $T:Y\to Y$ continuous and $\mu\in\mathcal P(Y)$, we define the pushforward measure $T_*\mu$ by
\begin{align*}
 (T_*\mu)(E)=\mu(T^{-1}E)
 \qquad(E\subseteq Y\text{ Borel}).
\end{align*}
$T_*\mu$ is again a probability measure and it satisfies, for $F\in C(Y)$ and $\mu \in\mathcal P(Y)$,
\begin{equation}\label{eq:dyn-pushforward-integral}
 \int_Y F\,d(T_*\mu)=\int_Y F\circ T\,d\mu.
\end{equation}
Moreover, the map $\mu\mapsto T_*\mu$ is weak-$*$ continuous. Finally, 
\begin{equation}\label{eq:dyn-invariance-tests}
 T_*\mu=\mu
 \quad\Longleftrightarrow\quad
 \int_Y(F\circ T-F)\,d\mu=0\quad\text{for every }F\in C(Y).
\end{equation}
In particular, $T_*\delta_y=\delta_{Ty}$. These identities and the invariance criterion are standard \cite[Equation~(B.4) and Lemma~B.12]{EinsiedlerWard2011}. 

We record the following result, which isolates the only use of unique ergodicity we will need in the proof of Theorem \ref{thm:joint-dynamical-limit}. We omit the proof as the result is a standard consequence of Proposition \ref{prop:dyn-measure-compactness} and \eqref{eq:dyn-invariance-tests}. 

\begin{lemma}
\label{lem:dyn-uniform-compactness}
Let $(Y,\mu,T_0,T_1)$ be uniquely ergodic. Let $\mathcal A\subseteq\mathbb N$ be unbounded. For each $x\in\mathcal A$, let $\Theta_x$ be a nonempty index set and let $\eta_{x,\theta}\in\mathcal P(Y)$ for
$\theta\in\Theta_x$. Suppose that, for $i=0,1$ and every $F\in C(Y)$,
\begin{equation}\label{eq:dyn-abstract-defect}
 \lim_{\substack{x\to\infty\\x\in\mathcal A}}
 \sup_{\theta\in\Theta_x}
 \left|\int_Y(F\circ T_i-F)\,d\eta_{x,\theta}\right|=0.
\end{equation}
Then, for every $F\in C(Y)$,
\begin{equation}\label{eq:dyn-abstract-conclusion}
 \lim_{\substack{x\to\infty\\x\in\mathcal A}}
 \sup_{\theta\in\Theta_x}
 \left|\int_Y F\,d\eta_{x,\theta}-\int_Y F\,d\mu\right|=0.
\end{equation}
\end{lemma}
\subsection{Proof of Theorem \ref{thm:joint-dynamical-limit}}\label{subsec:dyn-proof}

We first restate the arithmetic input supplied by Theorem \ref{thm:intro-main}. Recall that for $x \geq 1$, $(\nu_0,\nu_1)\in \{\omega,\Omega\}^2$ and $\mathbf m=(m_0,m_1)\in\mathbb Z^2$, we have the normalized joint count 
\begin{align*}
    P_x^{\nu_0, \, \nu_1}(\mathbf m)
 :=\frac1x \cdot \#\{n\le x:\nu_0(n)=m_0,\ \nu_1(n+1)=m_1\}.
\end{align*} It is nonnegative and satisfies $\sum_{\mathbf m\in\mathbb Z^2} P_x^{\nu_0, \, \nu_1}(\mathbf m) = 1$.
Theorem \ref{thm:intro-main} supplies an absolute $c>0$ and a set
$\mathcal X\subseteq\mathbb N$ with
\begin{equation}\label{eq:dyn-good-scales}
 \frac1{\log X}\sum_{\substack{x\le X\\x\notin\mathcal X}}\frac1x
 \ll(\log_2 X)^{-c},
\end{equation}
for $X$ sufficiently large, such that for $x \in \mathcal{X}$,
\begin{equation}\label{eq:dyn-arithmetic-input}
 \max_{\nu_0,\nu_1\in \{\omega,\Omega\}}
 \sup_{\mathbf m\in\mathbb Z^2}
 |P_x^{\nu_0, \, \nu_1}(\mathbf m)-G_{\log_2x}(\mathbf m)|
 \ll (\log_2x)^{-1-c}.
\end{equation}
For an absolutely summable function $w:\mathbb Z^2\to\mathbb C$, i.e. $\sum_{\mathbf m\in\mathbb Z^2}|w(\mathbf m)| < \infty$,
we write
\[
\|w\|_{\ell^1}
:=\sum_{\mathbf m\in\mathbb Z^2}|w(\mathbf m)|.
\]
Moreover, writing $\mathbf e_0=(1,0)$ and
$\mathbf e_1=(0,1)$, we define
\begin{align}
\label{eq:shift-defect}
V_i(w):=
\sum_{\mathbf m\in\mathbb Z^2}
|w(\mathbf m-\mathbf e_i)-w(\mathbf m)|,
\qquad i=0,1.
\end{align}
We need the following elementary estimate regarding the $\ell^1$ distance. 

\begin{lemma}
\label{lem:elem-estimate-shift-defect}
Let $p,q:\mathbb Z^2\to[0,\infty)$ satisfy
\[
\sum_{\mathbf m\in\mathbb Z^2}p(\mathbf m)=1,
\qquad
\sum_{\mathbf m\in\mathbb Z^2}q(\mathbf m)=S<\infty.
\]
Then, for every finite $Q\subseteq\mathbb Z^2$,
\[
\|p-q\|_{\ell^1}
\le
2\sum_{\mathbf m\in Q}|p(\mathbf m)-q(\mathbf m)|
+|1-S|
+2\sum_{\mathbf m \in Q^c}q(\mathbf m).
\]
\end{lemma}

\begin{proof}
Since $p$ has total mass $1$,
\begin{align*}
\sum_{\mathbf m\in Q^c}p(\mathbf m)
&\leq |1-S|+\sum_{\mathbf m\in Q^c}q(\mathbf m)+\sum_{\mathbf m\in Q}|p(\mathbf m)-q(\mathbf m)|.
\end{align*}
By nonnegativity and the triangle inequality,
\begin{align*}
\sum_{\mathbf m\in Q^c}|p(\mathbf m)-q(\mathbf m)|
&\le \sum_{\mathbf m\in Q^c}p(\mathbf m)+\sum_{\mathbf m\in Q^c}q(\mathbf m)\\
&\le |1-S|+2\sum_{\mathbf m\in Q^c}q(\mathbf m)+\sum_{\mathbf m\in Q}|p(\mathbf m)-q(\mathbf m)|.
\end{align*}
Adding $\sum_{\mathbf m\in Q}|p(\mathbf m)-q(\mathbf m)|$ on both sides yields the claim.
\end{proof}
The following proposition is  central to the proof of our dynamical theorem, and it is where we use the input obtained from Theorem \ref{thm:intro-main}.

\begin{proposition}
\label{prop:dyn-arithmetic-shifts}
Let $c>0$ be as in \eqref{eq:dyn-arithmetic-input} and put
$\kappa:=\min\{c/2,1/2\}>0$. As $x\to\infty$ along $\mathcal X$,
\begin{equation}\label{eq:dyn-l1-approximation}
 \max_{\nu_0,\nu_1\in \{\omega,\Omega\}}
 \|P_x^{\nu_0, \, \nu_1}-G_{\log_2x}\|_{\ell^1}
 \ll(\log_2x)^{-\kappa}.
\end{equation}
In particular, for $i=0,1$,
\begin{equation}\label{eq:dyn-weight-shifts}
 \max_{\nu_0,\nu_1\in \{\omega,\Omega\}}V_i(P_x^{\nu_0, \, \nu_1})
 \ll(\log_2x)^{-\kappa}.
\end{equation}
\end{proposition}

\begin{proof}
Fix $\delta:=\min\{c/4,1/8\}>0$ and $r:=(\log_2x)^{1/2+\delta}$. Define the set 
\begin{align*}
    Q(\log_2x,r):=\left\{\mathbf m=(m_0,m_1)\in\mathbb Z^2:
|m_i-\log_2x|\le r,\ i=0,1\right\}
\end{align*}
The interval $[\log_2x-r,\log_2x+r]$ contains at most $2r+2$ integers, so
\[
 \#Q(\log_2x,r)\le(2r+2)^2\ll (\log_2x)^{1+2\delta}.
\]
For $x\in\mathcal X$,
\eqref{eq:dyn-arithmetic-input} gives, uniformly in
$\nu_0, \nu_1 \in \{\omega, \Omega\}$,
\begin{equation}\label{eq:dyn-central-error}
 \sum_{\mathbf m\in Q(\log_2x,r)}
 |P_x^{\nu_0, \nu_1}(\mathbf m)-G_{\log_2x}(\mathbf m)|
 \ll (\log_2x)^{2\delta - c} \leq (\log_2x)^{-c/2}.
\end{equation}
The Gaussian estimates in \eqref{eq:gaussian-dynamical-1} and \eqref{eq:gaussian-dynamical-2} give
\[
 \left|1-\sum_{\mathbf m \in \Z^2}G_{\log_2x}(\mathbf m)\right|\ll (\log_2x)^{-1/2}, \, 
 \sum_{\mathbf m\notin Q(\log_2x,r)}G_{\log_2x}(\mathbf m)
 \ll\exp\left(-\frac14 (\log_2x)^{2\delta}\right).
\]
We apply Lemma \ref{lem:elem-estimate-shift-defect} with
$p=P_x^{\nu_0, \nu_1}$, $q=G_{\log_2x}$, and $Q=Q(\log_2x,r)$ to conclude that
\begin{align}
 \|P_x^{\nu_0, \nu_1}-G_{\log_2x}\|_{\ell^1}
 &\ll (\log_2x)^{-c/2}+(\log_2x)^{-1/2}
        +\exp\left(-\frac14 (\log_2x)^{2\delta}\right)
 \ll (\log_2x)^{-\kappa}.\label{eq:dyn-l1-before-shift}
\end{align}
This proves \eqref{eq:dyn-l1-approximation}.
For $i=0,1$, using the triangle inequality, we obtain:
\begin{align*}
 V_i(P_x^{\nu_0, \nu_1})
 &\le\sum_{\mathbf m\in\mathbb Z^2}
       |P_x^{\nu_0, \nu_1}(\mathbf m-\mathbf e_i)
          -G_{\log_2x}(\mathbf m-\mathbf e_i)|
       +V_i(G_{\log_2x})\\
       & \quad+\sum_{\mathbf m\in\mathbb Z^2}|G_{\log_2x}(\mathbf m)-P_x^{\nu_0, \nu_1}(\mathbf m)|\\
 &=2\|P_x^{\nu_0, \nu_1}-G_{\log_2x}\|_{\ell^1}+V_i(G_{\log_2x}).
\end{align*}
The last equality follows by the change of variables $\mathbf k=\mathbf m-\mathbf e_i$. Finally, using the estimate from \eqref{eq:dyn-l1-before-shift} and the bound given in \eqref{eq:gaussian-dynamical-3}, we conclude that
\[
V_i(P_x^{\nu_0, \nu_1})
\ll (\log_2x)^{-\kappa}+(\log_2x)^{-1/2}
\ll (\log_2x)^{-\kappa},
\]
since $\kappa\le 1/2$. This proves \eqref{eq:dyn-weight-shifts}.
\end{proof}
We now complete the proof of Theorem \ref{thm:joint-dynamical-limit}.

\begin{proof}[Proof of Theorem \ref{thm:joint-dynamical-limit}]
We have the uniquely ergodic $\N^2$-system $(Y, \mu, T_0, T_1$).
For an integer $x\ge1$, $\nu_0, \nu_1 \in \{\omega, \Omega\}$, and $y\in Y$, define the measures
\begin{equation}\label{eq:dyn-orbit-measures}
 \eta_{x,y}^{\nu_0, \nu_1}
 :=\frac1x\sum_{n\le x}
       \delta_{T_0^{\nu_0(n)}T_1^{\nu_1(n+1)}y}
 =\sum_{\mathbf m\in\mathbb N_0^2}
       P_x^{\nu_0, \nu_1}(\mathbf m)\delta_{T^{\mathbf m}y}.
\end{equation}
The sums are finite for $x$ fixed. In particular, for $F \in C(Y)$,
\begin{equation}\label{eq:dyn-average-as-integral}
 \frac1x\sum_{n\le x}
 F\bigl(T_0^{\nu_0(n)}T_1^{\nu_1(n+1)}y\bigr)
 =\int_Y F\,d\eta_{x,y}^{\nu_0, \nu_1}.
\end{equation}
Recall that a finite signed Borel measure is a difference of two
finite positive Borel measures. For such a measure
$\sigma$ on $Y$, we define its total variation norm by
\begin{align*}
\|\sigma\|_{\mathrm{TV}}
:=
\sup_{\|h\|_\infty\le 1}
\left|\int_Y h\,d\sigma\right|,
\end{align*}
where the supremum is taken over all bounded
Borel measurable functions $h:Y\to\mathbb C$. In particular, for any absolutely summable family of real numbers $(a_{\mathbf m})$ and points $z_{\mathbf m}\in Y$, we have
\begin{align}
\label{eq:dyn-atomic-tv}
\left\|
\sum_{\mathbf m \in \Z^2}a_{\mathbf m}\delta_{z_{\mathbf m}}
\right\|_{\mathrm{TV}}
\le
\sum_{\mathbf m \in \Z^2}|a_{\mathbf m}|.
\end{align}
This follows directly from the definition and
the triangle inequality.

We now fix $i\in\{0,1\}$. By commutativity,
$T_iT^{\mathbf m}=T^{\mathbf m+\mathbf e_i}$. The pushforward identity $T_*\delta_y=\delta_{Ty}$ therefore gives
\[
 (T_i)_*\eta_{x,y}^{\nu_0, \nu_1}
 =\sum_{\mathbf m\in\mathbb N_0^2}
       P_x^{\nu_0, \nu_1}(\mathbf m)
       \delta_{T^{\mathbf m+\mathbf e_i}y}
 =\sum_{\mathbf m\in\mathbb N_0^2}
       P_x^{\nu_0, \nu_1}(\mathbf m-\mathbf e_i)
       \delta_{T^{\mathbf m}y}.
\]
The second equality follows by a change of variables, using the fact that $P_x^{\nu_0, \nu_1}$ vanishes outside of
$\mathbb N_0^2$. Subtracting \eqref{eq:dyn-orbit-measures} on both sides gives
\begin{equation}\label{eq:dyn-pushforward-difference}
 (T_i)_*\eta_{x,y}^{\nu_0, \nu_1}-\eta_{x,y}^{\nu_0, \nu_1}
 =\sum_{\mathbf m\in\mathbb N_0^2}
   \bigl(P_x^{\nu_0, \nu_1}(\mathbf m-\mathbf e_i)
         -P_x^{\nu_0, \nu_1}(\mathbf m)\bigr)
     \delta_{T^{\mathbf m}y}.
\end{equation}
By the total variation inequality in \eqref{eq:dyn-atomic-tv}, it follows that
\begin{equation}\label{eq:dyn-orbit-defect-TV}
 \|(T_i)_*\eta_{x,y}^{\nu_0, \nu_1}-\eta_{x,y}^{\nu_0, \nu_1}
       \|_{\mathrm{TV}}
 \le V_i(P_x^{\nu_0, \nu_1}).
\end{equation}
In particular, this gives that for every $F\in C(Y)$,
\begin{equation}\label{eq:dyn-orbit-defect-tests}
 \left|\int_Y(F\circ T_i-F)\,d\eta_{x,y}^{\nu_0, \nu_1}\right|
 \le\|F\|_\infty V_i(P_x^{\nu_0, \nu_1}).
\end{equation}
Proposition~\ref{prop:dyn-arithmetic-shifts} shows that the right-hand side tends to zero uniformly along $\mathcal X$.

It remains to identify the limiting measures using unique ergodicity, which we accomplish by applying Lemma \ref{lem:dyn-uniform-compactness} with $\mathcal A=\mathcal X$, $\Theta_x=\{\omega,\Omega\}^2\times Y$ and the family of measures $\eta_{x,y}^{\nu_0,\nu_1}$. Indeed, by \eqref{eq:dyn-orbit-defect-tests} and Proposition \ref{prop:dyn-arithmetic-shifts},
for every $F\in C(Y)$ and $i=0,1$,
\begin{align*}
&\max_{\nu_0,\nu_1\in\{\omega,\Omega\}}
\sup_{y\in Y}
\left|
\int_Y(F\circ T_i-F)
\,d\eta_{x,y}^{\nu_0,\nu_1}
\right|\le
\|F\|_\infty
\max_{\nu_0,\nu_1\in\{\omega,\Omega\}}
V_i(P_x^{\nu_0,\nu_1})
\longrightarrow 0,
\end{align*}
as $x\to\infty$ along $\mathcal X$.
Thus, the hypotheses of Lemma \ref{lem:dyn-uniform-compactness} are satisfied, and we conclude that
\begin{align}
\max_{\nu_0,\nu_1\in\{\omega,\Omega\}}
\sup_{y\in Y}
\left|
\int_Y F\,d\eta_{x,y}^{\nu_0,\nu_1}
-\int_Y F\,d\mu
\right|
\longrightarrow 0.
\end{align}
Recalling \eqref{eq:dyn-orbit-measures}, this establishes the
first claim asserted in Theorem \ref{thm:joint-dynamical-limit}. The logarithmically averaged conclusion \eqref{eq:intro-joint-log-dynamics} follows by a straightforward argument using partial summation.
\end{proof}
\printbibliography[title={References}]

\end{document}